\documentclass[12pt,reqno]{amsart}
\usepackage{latexsym,amsmath,amsfonts,amscd,amssymb}
\usepackage{graphics}

\usepackage{tikz-cd}

\newcommand{\inc}{\hookrightarrow}
\newcommand{\la}{\langle}
\newcommand{\ra}{\rangle}
\newcommand{\x}{\times}

\newcommand{\too}{\longrightarrow}
\newcommand{\bd}{\partial}

\newcommand{\SD}{{\mathcal{D}}}

\newcommand{\ZZ}{\mathbb{Z}}
\newcommand{\CC}{\mathbb{C}}

\newcommand{\CP}{\mathbb{C}P}
\newcommand{\RR}{\mathbb{R}}
\newcommand{\TT}{\mathbb{T}}
\newcommand{\QQ}{\mathbb{Q}}

\DeclareMathOperator{\lcm}{lcm}

\DeclareMathOperator{\Id}{Id}

\newcommand{\orb}{\scriptsize \mathrm{orb}}

\renewcommand{\a}{\alpha}
\renewcommand{\b}{\beta}

\newcommand{\g}{\gamma}
\newcommand{\e}{\epsilon}
\newcommand{\f}{\varphi}

\renewcommand{\l}{\lambda}

\newcommand{\s}{\sigma}

\renewcommand{\o}{\omega}
\newcommand{\p}{\phi}

\newcommand{\G}{\Gamma}

\renewcommand{\S}{\Sigma}
\newcommand{\D}{\Delta}

\newcommand{\cO}{\mathcal O}
\newcommand{\PP}{\mathbb P}

\newcommand{\set}[1]{\left\{#1\right\}}

\newtheorem{theorem}{Theorem}
\newtheorem{proposition}[theorem]{Proposition}
\newtheorem{lemma}[theorem]{Lemma}
\newtheorem{definition}[theorem]{Definition}

\newtheorem{corollary}[theorem]{Corollary}
\newtheorem{remark}[theorem]{Remark}
\newtheorem{conjecture}[theorem]{Conjecture}

\title[K-contact and not Sasakian simply connected $5$-manifold]{A K-contact 
simply connected $5$-manifold with no semi-regular Sasakian structure}

\author[A. Ca\~{n}as]{Alejandro Ca\~{n}as}
\address{Departamento de \'Algebra, Geometr\'{\i}a y Topolog\'{\i}a, Universidad de M\'alaga, 
Campus de Teatinos, s/n, 29071 M\'alaga, Spain}
\email{alejandro.cm.95@uma.es}

\author[V. Mu\~{n}oz]{Vicente Mu\~{n}oz}
\address{Departamento de \'Algebra, Geometr\'{\i}a y Topolog\'{\i}a, Universidad de M\'alaga, 
Campus de Teatinos, s/n, 29071 M\'alaga, Spain}
\email{vicente.munoz@uma.es}

\author[J. Rojo]{Juan Rojo}
\address{ETSI Ingenieros Inform\'aticos, Universidad Polit\'ecnica de Madrid,
Campus de Montegancedo, 28660, Madrid, Spain}
\email{juan.rojo.carulli@upm.es}

\author[A. Viruel]{Antonio Viruel}
\address{Departamento de \'Algebra, Geometr\'{\i}a y Topolog\'{\i}a, Universidad de M\'alaga, 
Campus de Teatinos, s/n, 29071 M\'alaga, Spain}
\email{viruel@uma.es}



\begin{document}

\begin{abstract}
  We construct the first example of a $5$-dimensional simply connected compact manifold 
  that admits a K-contact structure but does not admit any semi-regular Sasakian structure.
  For this, we need two ingredients: (a) to construct a suitable simply connected symplectic 
  $4$-manifold with disjoint symplectic surfaces spanning the homology,
  all of them of genus $1$ except for one
  of genus $g>1$ ; (b) to prove a bound on the 
  second Betti number $b_2$ of an algebraic surface with $b_1=0$ and having 
  disjoint complex curves spanning the homology, all of them of genus $1$ except for one
  of genus $g>1$. 
\end{abstract}

\maketitle

\section{Introduction}\label{sec:1}

In geometry, it is a central question to determine when a given manifold admits an specific geometric structure. 
Complex geometry provides numerous examples of compact manifolds with rich topology, and there is a number
of topological properties that have to be satisfied by complex manifolds. For instance, compact K\"ahler manifolds
satisfy strong topologial properties like the hard Lefschetz property, the formality of its rational homotopy type \cite{DGMS}, or
restrictions on the fundamental group \cite{ABCKT}. A natural approach is to weaken the given structure and to ask
to what extent a manifold having the weaker structure may admit the stronger one. In the case of K\"ahler manifolds,
if we forget about the integrability of the complex structure, then we are dealing with symplectic manifolds. There has
been enormous interest in the construction of (compact) symplectic manifolds that do not admit K\"ahler structures,
and in determining its topological properties \cite{OT}. In dimension $4$, when we deal with complex
surfaces, we have the Enriques-Kodaira classification \cite{BPV} that helps in the understanding of this question.

In odd dimension, Sasakian and K-contact manifolds are analogues of K\"ahler and symplectic manifolds, respectively. 
Sasakian geometry has become an important and active subject, especially after the appearance of
the fundamental treatise of Boyer and Galicki \cite{BG}. Chapter 7 of this book contains an extended discussion of 
topological problems of Sasakian and K-contact manifolds. 

The precise definition of the structures that we are dealing with in this paper is as follows.
Let $M$ be a smooth manifold. A {\em K-contact structure} on $M$ consists of
tensors $(\eta,J)$ such that $\eta$ is a contact form $\eta\in \Omega^1(M)$, i.e. $\eta\wedge (d\eta)^n>0$ everywhere,
and $J$ is an endomorphism of $TM$ such that: 
 \begin{itemize}
\item $J^2=-\Id + \xi\otimes\eta$, where $\xi$ is the Reeb
vector field of $\eta$, $i_\xi \eta=1$, $i_\xi (d\eta)=0$,
\item $d\eta (J X,J Y)\,=\,d\eta (X,Y)$, for all vector fields $X,Y$, and
$d\eta (J X,X)>0$ for all nonzero $X\in \ker \eta$, 
\item the Reeb field $\xi$ is Killing with respect to the
Riemannian metric $g(X,Y)\,=\,d\eta (J X,Y)+\eta (X)\eta(Y)$.
\end{itemize}
We may denote $(\eta,J)$ or $(\eta,J,g,\xi)$ a K-contact structure on $M$,
since $g$ and $\xi$ are in fact determined by $\eta$ and $J$.
Note that the endomorphism $J$ defines a complex structure on $\SD=\ker \eta$ compatible with $d\eta$, 
hence $J$ is orthogonal with respect to the metric $g|_\SD$. By definition, the Reeb vector
field $\xi$ is orthogonal to $\SD$.
Finally, a \emph{K-contact manifold} is $(M,\eta, J ,g, \xi)$, a manifold $M$ endowed with a K-contact structure.
For a K-contact manifold $M$, the condition that the Reeb vector field be Killing with respect to the metric $g$
is very rigid and it imposes strong constraints on the topology.
In particular, it is not difficult to find manifolds that admit contact but do not admit K-contact structures
in any odd dimension, for instance the odd dimensional tori, see \cite[Corollary 7.4.2]{BG}.
For simply-connected $5$-manifolds (i.e.\ Smale-Barden manifolds), 
one can also find infinitely many of them admitting contact but not K-contact
structures, see \cite[Theorem 10.2.9 and Corolary 10.2.11]{BG}.

Just as for almost complex structures, there is the notion of
integrability of a K-contact structure. More precisely, a K-contact
structure $(\eta,J,g,\xi)$ is called \emph{normal} if the Nijenhuis
tensor $N_{J}$ associated to the tensor field $J$, defined by
 \begin{equation*}\label{NPhi}
 N_{J} (X, Y) = {J}^2 [X, Y] + [J X, J Y] -  J [ J X, Y] - J [X, J Y]\, ,
 \end{equation*}
satisfies the equation $N_{J} \,=\,-{d}\eta\otimes \xi$.
A {\it Sasakian structure} on $M$ is a normal K-contact structure $(\eta, J ,g, \xi)$, 
and we call $(M ,\eta, J ,g, \xi)$ a {\em Sasakian manifold}

Let $(M,\eta,J,g)$ be a K-contact manifold. Consider the cone as the Riemannian manifold
$C(M)=(M\times\mathbb{R}_{+},t^2g+dt^2)$.
One defines an almost complex structure $I$ on $C(M)$ by:
 \begin{itemize}
\item $I(X)=J(X)$ on $\ker\eta$,
\item $I(\xi)=t{\partial\over\partial t},\,I \left(t{\partial\over\partial t}\right)=-\xi$, 
for the Reeb vector field $\xi$ of $\eta$.
\end{itemize}
Then $(M,\eta,J,g)$ is Sasakian if and only if $I$ is integrable, that is, if $(C(M),I,t^2g + dt^2)$ is a K\"ahler manifold,
see \cite[Definition 6.5.15]{BG}.

Slightly abusing notation, if we are given a smooth manifold $M$ with no specified contact structure, 
we will say that $M$ is K-contact (Sasakian) if it admits some K-contact (Sasakian) structure.
In this paper we will mainly be concerned with geography questions, i.e.\ which smooth manifolds admit K-contact or Sasakian structures.

In dimension $3$, every K-contact manifold admits a Sasakian structure \cite{JK}.
For dimension greater than $3$, there is much interest on constructing K-contact manifolds 
which do not admit Sasakian structures.
The odd Betti numbers up to degree $n$ of Sasakian $(2n+1)$-manifolds must be even. 
The parity of $b_1$ was used to produce the first examples of K-contact manifolds with no Sasakian structure
\cite[example 7.4.16]{BG}. In the case of even Betti numbers, more refined tools are needed 
to distinguish K-contact from Sasakian manifolds. 
The  cohomology algebra of a Sasakian manifold satisfies a hard Lefschetz property \cite{CNY}. 
Using it examples of K-contact non-Sasakian manifolds are produced in \cite{CNMY} in dimensions $5$ and $7$. 
These examples are nilmanifolds with even Betti numbers, so in particular they are not simply connected.

When one moves to simply connected manifolds, K-contact non-Sasakian 
examples of any dimension $\geq 9$ were constructed in \cite{HT} using the evenness of $b_3$
of a compact Sasakian manifold. Alternatively, using the hard Lefschetz property
for Sasakian manifolds there are examples \cite{Lin} of simply connected K-contact non-Sasakian
manifolds of any dimension $\geq 9$.
In \cite{BFMT,Tievsky} the rational homotopy type of Sasakian manifolds is studied.
All higher order Massey products for simply connected Sasakian manifolds vanish,
although there are Sasakian manifolds with non-vanishing triple Massey products \cite{BFMT}.
This yields examples of simply connected K-contact non-Sasakian manifolds in dimensions $\geq 17$. 
However, Massey products are not suitable for the analysis of lower dimensional manifolds.

The problem of the existence of simply connected K-contact non-Sasakian compact manifolds 
(open problem 7.4.1 in \cite{BG}) is still open in dimension $5$. It was solved for dimensions $\geq 9$  
in \cite{CNY,CNMY,HT} and for dimension $7$ in \cite{MT} by a combination of various techniques 
based on the homotopy theory and symplectic geometry. In the least possible dimension the problem appears to 
be much more difficult. 
A simply connected compact oriented $5$-manifold is called a {\it Smale-Barden manifold}. 
These manifolds are classified topologically by $H_2(M,\mathbb{Z})$ and the second Stiefel-Whitney class, 
see \cite{Barden, Smale} for the classification by Smale and Barden.
Chapter 10 of the book \cite{BG} by Boyer and Galicki is devoted to a description of some Smale-Barden manifolds which 
carry Sasakian structures. The following problem is still open \cite[open problem 10.2.1]{BG}. 

{\it Do there exist Smale-Barden manifolds which carry K-contact but do not carry Sasakian structures?}

In the pioneering work \cite{MRT}, it is done a first step towards a positive answer to the question.
A homology Smale-Barden manifold is a compact $5$-dimensional manifold with $H_1(M,\mathbb{Z})=0$. 
A Sasakian structure is \emph{regular} if the leaves of the Reeb flow are a folation by circles with the
structure of a circle bundle over a smooth manifold. The Sasakian structure is \emph{quasi-regular} if the foliation
is a Seifert circle bundle over a (cyclic) orbifold, and it is \emph{semi-regular} if the base orbifold has only locus of non-trivial isotropy of codimension $2$, i.e.\ its underlying space is a topological manifold. 
Recall that the isotropy locus of an orbifold is the subset of points with non-trivial isotropy group.
It is a remarkable
result, although not difficult, that any manifold admitting a Sasakian structure has also a quasi-regular Sasakian structure (in any odd dimension).
Therefore, a Sasakian manifold is a Seifert bundle over a cyclic K\"ahler orbifold \cite{MRT}.
 
Correspondingly, for K-contact manifolds we also define regular, quasi-regular and semi-regular K-contact structures
with the same conditions. Any K-contact manifold admits a quasi-regular K-contact structure by \cite[Theorem 7.1.10]{BG} and
\cite{Ruk}.
Hence, a K-contact manifold is a Seifert bundle over a cyclic symplectic orbifold. 
Such orbifold has a isotropy locus which is a (stratified) collection of symplectic sub-orbifolds. 
The K-contact structure is semi-regular if the
symplectic orbifold has isotropy locus of codimension $2$. The main result of \cite{MRT} is: 

\begin{theorem}[\cite{MRT}] \label{thm:11}
There exists a homology Smale-Barden manifold which admits a semi-regular K-contact structure but which 
does not carry any semi-regular Sasakian structure.
\end{theorem}

The construction of \cite{MRT} relies upon subtle obstructions to admit Sasakian structures in 
dimension $5$ found by Koll\'ar \cite{Kollar}. If a $5$-dimensional manifold $M$  has a Sasakian structure, then 
it is a Seifert bundle over a K\"ahler orbifold $X$ with isotropy locus a collection of complex curves $D_i$
with isotropy (multiplicity) $m_i$. We
have the following topological characterization of the homology of $M$ in terms of that of $X$.

\begin{theorem}[{\cite[Theorem 16]{MRT}}] \label{thm:12}
Suppose that $\pi:M\to X$ is a semi-regular Seifert bundle with isotropy surfaces $D_i$ with multiplicities $m_i$. 
Then $H_1(M,\ZZ)=0$ if and only if
 \begin{enumerate}
 \item $H_1(X,\ZZ)=0$;
 \item the map $H^2(X,\ZZ)\to \oplus_i H^2(D_i,\ZZ/m_i)$  induced by the inclusions $D_i \subset X$, is surjective;
 \item the Chern class $c_1(M/e^{2\pi i/\mu}) \in H^2(X,\ZZ)$ 
 of the circle bundle $M/e^{2\pi i/\mu}$ is a primitive element, where $\mu$ is the lcm of all $m_i$.
 \end{enumerate}
 Moreover, $H_2(M,\ZZ)=\ZZ^k\oplus \bigoplus (\ZZ/m_i)^{2g_i}$, $g_i=$genus of $D_i$, $k+1=b_2(X)$.
\end{theorem}
Recall that an element $x$ of a $\ZZ$-module is called \emph{primitive} if it
is not of the form $x=ny$ for some integer $n>1$.

\begin{corollary}[{\cite[Corollary 18]{MRT}}] \label{cor:18}
Suppose that $M$ is a $5$-manifold with $H_1(M,\ZZ)=0$ and 
$H_2(M,\ZZ)=\ZZ^k\oplus \bigoplus_{i=1}^{k+1} (\ZZ/p^i)^{2g_i}$, $k\geq 0$, $p$ a prime, and $g_i\geq 1$. 
If $M\to X$ is a semi-regular Seifert bundle, then $H_1(X,\ZZ)=0$, $H_2(X,\ZZ)=\ZZ^{k+1}$, and 
the ramification locus has $k+1$ disjoint surfaces $D_i$ linearly independent in rational homology, and
of genus $g(D_i)= g_i$.
\end{corollary}

In \cite[Theorem 23]{MRT}, the authors construct a symplectic $4$-dimensional orbifold with 
disjoint symplectic surfaces spanning the second homology. This is the first example of such phenomenon and has $b_2=36$.
The genera of the isotropy surfaces satisfy $1\leq g_i\leq 3$, with several of them with genus $3$.
Using this symplectic orbifold $X$, we obtain a semi-regular K-contact $5$-manifold $M$ with 
 \begin{equation}\label{eqn:homol}
H_1(M,\ZZ)=0, \qquad H_2(M,\ZZ)=\ZZ^{35}\oplus \bigoplus_{i=1}^{36} (\ZZ/p^i)^{2g_i}.
 \end{equation}

For understanding the Sasakian side, in \cite{MRT} it is proved the following result

\begin{theorem}[{\cite[Theorem 32]{MRT}}] \label{thm:MRT1}
Let $S$ be a smooth K\"ahler surface with $H_1(S,\QQ)=0$ and 
containing $D_1,\ldots, D_b$, $b=b_2(S)$, smooth disjoint complex curves with $g(D_i)=g_i>0$, and spanning $H_2(S,\QQ)$. Assume that: 
 \begin{itemize}
 \item at least two $g_i$ are $>1$,
 \item $1\leq g_i\leq 3$.
 \end{itemize}
Then $b\leq 2\max\{g_i\}+3$.
\end{theorem}

As a Corollary \cite[Proposition 31]{MRT}, there is no Sasakian semi-regular $5$-dimensional manifold with 
homology given by (\ref{eqn:homol}). So $M$ is $K$-contact, but does not admit any semi-regular Sasakian structure,
proving Theorem \ref{thm:11}.

Theorem \ref{thm:MRT1} is a result in accordance with the following conjecture
from \cite{MRT}:

\begin{conjecture} \label{conj:Mu}
  There does not exist a K\"ahler manifold or a K\"ahler orbifold $X$ with $b_1=0$ and with
  $b_2 \geq 2$ having disjoint complex curves spanning $H_2(X,\QQ)$, all of genus $g\geq 1$.
\end{conjecture}

The present work enhances the main result from \cite{MRT} given in Theorem \ref{thm:11}, to 
achieve a $5$-manifold that it is furthermore simply connected.
Our main result is the following:

\begin{theorem} \label{thm:main}
There exists a (simply connected) Smale-Barden manifold which admits a semi-regular K-contact structure but which 
does not carry any semi-regular Sasakian structure.
\end{theorem}

On the one hand, we provide a new construction of a symplectic $4$-manifold $X$ with 
$b_1 = 0$ and $b_2 = b > 1$, having a collection of disjoint symplectic surfaces $C_1, \ldots, 
C_b$ spanning $H_2(X,\QQ)$, and all with genus $g_i \geq 1$. This is based on the following
phenomenom which can be performed in the symplectic setting but not in the algebro-geometric
situation. 

Start with the complex projective plane $\CP^2$ and two generic (smooth) complex cubic curves $C_1, C_2$. 
Note $C_1$ and $C_2$ have genus $1$ by the genus-degree formula,
and they intersect in nine points $P_1, \ldots, P_9$. A third complex cubic curve passing through $P_1,\ldots, P_8$ 
has to go necessarily through $P_9$. This is a purely algebraic phenomenon. However, it is possible to 
construct a third symplectic cubic $C_3$ going through $P_1, \ldots , P_8$, but intersecting $C_1$ at another point $P_{10}$, 
and $C_2$ at a different point $P_{11}$. Note that each $C_i$ misses exactly one of the eleven points $P_1,\ldots, P_{11}$. 
Looking at this more symmetrically, we aim to have a collection of $11$ points $\D = \{P_1, \ldots, P_{11}\}$ 
and $11$ cubic complex curves $C_1, \ldots, C_{11}$ such that $C_i$ passes through the points of $\D - \{P_i\}$, $i=1,\ldots, 11$.
In this way, the intersections are $C_i \cap  C_j = \D - \{P_i, P_j\}$ and no more points. Blowing up at all points of $\D$,
we get the (symplectic) $4$-manifold $X = \CP^2 \# 11 \overline{\CP}{}^2$, with $11$ complex curves of genus $1$ and disjoint. 
An extra (complex) curve can be obtained by taking a singular complex curve $G$ of degree $10$ 
with ordinary triple points at the points of $\D$.  Note that $G$ has genus $3$ by the Pl\"ucker formulas.
Moreover, as $G\cdot C_i=30$ equals the geometric intersection, that is, $3$ times
for each of the $10$ triple points in $G\cap C_i=\D-\{P_i\}$, we would not have more intersections. This curve is
of genus $g_G=3$, and it becomes a smooth genus $3$ curve in the blow-up, that is disjoint from the others.
This heuristic argument has to be carried out in a slightly different guise, by making a symplectic construction in a 
tubular neighbourhood of a cubic curve and a complex line and gluing it in symplectically (see Section \ref{sec:2}). 

\begin{theorem} \label{thm:main1}
Let $P_1,\ldots, P_{11}$ be $11$ points in $\CP^2$. 
Then there exist symplectic surfaces 
 $$
 C_1,C_2,\ldots, C_{11},G \subset \CP^2
 $$
such that: 
 \begin{itemize}
 \item $C_i$ is a genus $1$ smooth surface and $P_j \in C_i$ for $j\neq i$, $P_i \notin C_i$.
 \item The surfaces $C_i, C_j$, $i \neq j$ intersect exactly at 
 $\{P_1, \ldots , P_{11} \}-\{P_i,P_j\}$, positively and transversely.
 \item $G$ is a genus $3$ singular symplectic surface whose only singularities
 are $11$ triple points at $P_i$ (with different branches intersecting positively).
 Moreover $G$ intersects each $C_i$ only at the points $P_j$, $j\neq i$, 
 and all the intersections of $C_i$ with the branches of $G$ are positive and transverse.
 \end{itemize}
\end{theorem}

Using this, we construct our K-contact $5$-manifold. First we blow up $\CP^2$ at the $11$ points $P_1,\ldots, P_{11}$,
to obtain a symplectic manifold, which topologically is $X=\CP^2 \# 11 \overline{\CP}{}^2$. The proper transforms
of $C_1,\ldots, C_{11},G$ are symplectic surfaces in $X$, via the method in \cite[section 5.2]{MRT}. The proper transform of $G$
becomes a smooth genus $3$ symplectic surface. Therefore $b_2(X)=12$ and it has $12$ disjoint symplectic surfaces,
$11$ of them of genus $g_i=1$ and one of genus $g_{12}=3$.
Take numbers $m_i$. Using \cite[Proposition 7]{MRT}, we make $X$ into an orbifold $X'$ whose isotropy locus
is $C_i$ with multiplicity $m_i$ and $G$ with multiplicity $m_{12}$. 
Then we can take a Seifert bundle $M \to X'$ with primitive Chern class $c_1(M/e^{2\pi i/\mu})=[\omega]$ 
after a small perturbation of the symplectic form,
as in \cite[Lemma 20]{MRT}. The manifold $M$ is K-contact and has
 \begin{equation}\label{eqn:homol2}
 H_1(M,\ZZ)=0, \qquad H_2(M,\ZZ)=\ZZ^{11}\oplus \bigoplus_{i=1}^{12} (\ZZ/m_i)^{2g_i}.
 \end{equation}
We choose a prime $p$ and $m_i=p^i$, so that all $m_i$ are distinct and pairwise non-coprime.

Given a Seifert bundle $M\to X'$, the fundamental group of $M$ is directly related to the \emph{orbifold fundamental group} of $X'$ by
the long exact sequence
 $$
 \ldots \to \pi_1(S^1)=\ZZ \to \pi_1(M) \to \pi_1^{\orb}(X') \to 1
 $$
When $\pi_1^{\orb}(X')=1$, we have that $\pi_1(M)$ is abelian, and hence if $H_1(M,\ZZ)=0$ then $M$ is simply connected.
We prove in Section \ref{sec:4} the following

\begin{theorem} \label{thm:main2}
 For the orbifold $X'$ constructed above, $\pi_1^{\orb}(X')=1$. Hence $M$ is a Smale-Barden manifold.
 \end{theorem}

On the second hand, we have to prove that $M$ cannot admit a semi-regular Sasakian structure. If this were the
case, then there would be a Seifert bundle $M\to Y$, where $Y$ is a K\"ahler orbifold. By \cite[Proposition 10]{MRT},
this orbifold $Y$ is a complex manifold, and as the Sasakian structure is semi-regular, $Y$ is smooth. As the homology
of $M$ is given by (\ref{eqn:homol2}), then Corollary \ref{cor:18} guarantees that $Y$ has $b_1=0$, $b_2=12$
and contains $12$ disjoint smooth complex curves $C_1',\ldots, C_{11}',G'$, where $g(C_i')=1$ and $g(G')=3$.
We prove the corresponding instance of Conjecture \ref{conj:Mu}. Note that this is not covered by Theorem \ref{thm:MRT1}.

\begin{theorem}\label{thm:conj2}
Let $S$ be a smooth complex surface with $H_1(S,\QQ)=0$ and 
containing $D_1,\ldots, D_b$, $b=b_2(S)$, smooth disjoint complex curves with
genus $g(D_i)=g_i>0$, and spanning $H_2(S,\QQ)$. Assume that
$g_i=1$, for $1\leq i\leq b-1$. Then $b\leq 2g^2_b-4g_b+3$.
\end{theorem}

In particular, the case $b_2=12$, $g_i=1$, for $1\leq i\leq 11$ and $g_{12}=3$ cannot happen.

\begin{corollary} \label{cor:spin} 
Let $M$ be a $5$-dimensional manifold with $H_1(M,\ZZ)=0$ and 
$$
H_2(M,\ZZ)= \ZZ^{11} \oplus \bigoplus_{i=1}^{12} (\ZZ/p^{i})^{2g_i} \, ,
$$
where $g_i=1$ for $1 \le i \le 11$, $g_{12}=3$, and $p$ is a prime number.
Then $M$ does not admit a semi-regular Sasakian structure.
\end{corollary}

This proves Theorem \ref{thm:main}. It remains to see Theorems \ref{thm:main1}, \ref{thm:main2} and \ref{thm:conj2}.
We prove Theorem \ref{thm:main1} in Section \ref{sec:3}, Theorem \ref{thm:main2} in Section \ref{sec:4}
and Theorem \ref{thm:conj2} in Section \ref{sec:5}.

The manifold $M$ in Corollary \ref{cor:spin} is spin if $p=2$, and can be chosen to be spin or non-spin
if $p>2$.

Both here and in \cite{MRT} we have provided the first examples of symplectic $4$-manifolds
containing symplectic surfaces of positive genus and spanning the homology. Whereas the example of \cite{MRT}
is a symplectic $4$-manifold that does not admit a complex structure (see Remark \ref{rem:Roger}), the manifold
constructed here, $X=\CP^2\# 11 \overline{\CP}{}^2$ does admit a Kähler structure. So $X$ is symplectic deformation
equivalent to a Kähler manifold, but the $12$ symplectic surfaces inside it cannot be deformed to complex curves in the way.
We thank Roger Casals from prompting this question to us.

\noindent\textbf{Acknowledgements.} 
We are grateful to Enrique Arrondo, Alex Tralle, Fran Presas and Roger Casals for useful comments.
Thanks to the two anonymous referees that have given us numerous comments.
Second author partially supported by Project MINECO (Spain) PGC2018-095448-B-I00. Fourth author partially supported
by Project MINECO (Spain) MTM2016-78647-P.

\section{Symplectic plumbing}\label{sec:2}

The specific aim of this section is to give suitable local models for a small neighborhood 
of a union of two positively intersecting symplectic surfaces inside a $4$-manifold.
See references \cite{Gay-Stip, Gay-Mark} for related content.

\subsection{Definition of symplectic plumbing} \label{sec:plumbing}

Let $(S,\o)$ be a compact symplectic surface and $\pi:E\to S$ be a complex line bundle. Topologically, 
$E$ is determined by the Chern class $d=c_1(E)$ which is
the self-intersection of $S$ inside $E$, $d=[S]^2$. We put a hermitian structure in $E$,
so we can define a neighbourhood via a disc bundle of some fixed radius $c>0$,
denoted by $B_c(S)\subset E$. We construct a symplectic form on $B_c(S)$ next.
First, we write  $V'\Subset V$ if $V'$ is an open subset such that its closure
$\overline{V'}\subset V$.

\begin{lemma} \label{lem:2.1}
For small enough $c>0$, $B_c(S)$  admits a symplectic form $\o_E$ 
which is compatible with the complex structure of the fibers of the complex line bundle,
and such that the inclusion $(S,\o) \hookrightarrow (B_c(S),\o_E)$
is symplectic. If $V\subset S$ is a trivializing open set, $E|_V\cong V\x \CC$, and
$V'\Subset V$, we can arrange so that $\o_E|_{B_c(S)\cap E|_{V'}}$ is
the symplectic product structure on $B_c(S)\cap E|_{V'}\cong V'\x B_c(0)$,
with $B_c(0) \subset \CC$ a ball centered at $0$.
\end{lemma}

\begin{proof}
Take $S=\bigcup_\a U_\a$ a cover of $S$, with each $U_\a$ symplectomorphic to a ball,
and trivializations  $E|_{U_\a} \cong U_\a \x \CC$.
In the fiber $\CC$ we put coordinates $u+iv$ and consider the standard symplectic form
$\o_0=du \wedge dv= d(u dv)=d \eta$.
Denote $\varpi_\a: U_\a \x \CC \to \CC$ the projection over the second factor,
and take $\rho_\a$ a smooth partition of unity subordinated to the cover $U_\a$ of $S$.
Define
 $$
 \o_E= \pi^* \o_S + \sum_\a d( (\pi^* \rho_\a) \cdot (\varpi_\a^* \eta)).
 $$
For $x\in S$, we have $\o_E|_{E_x} =\sum_\a \rho_\a(x) \omega_0=\omega_0$, using that the
changes of trivializations preserve $\o_0$. Then using the decomposition $T_xE = T_xS \oplus E_x$,
we have that $(\o_E)^2(x) = \o_S(x) \wedge \o_0 >0$. Therefore $\o_E$ is symplectic on 
the zero section $S \subset E$.
Since this is an open condition, it holds in some neighborhood $B_c(S)$ of the zero section.

For the last part, just take an open cover $U_\a$ of $S-V'$ together with $V$ in the construction
above.
\end{proof}

The submanifold $S\subset E$ and
any fiber $E_x\subset E$ are symplectic, and they are symplectically orthogonal.

\smallskip

Now we move to the definition of plumbing as a symplectic neighbourhood of the union of
two intersecting symplectic surfaces $S_1,S_2$. Take points $P_1,\ldots, P_m\in S_1$
and $Q_1,\ldots, Q_m\in S_2$. We define 
 $$
 S=S_1\sqcup S_2/ P_i\sim Q_i, i=1,\ldots, m
  $$
and we can write $S=S_1\cup S_2$. Let now $E_1 \to S_1$ and $E_2 \to S_2$ be two
complex line bundles, where $d_i=c_1(E_i)$ which is
the self-intersection of $S_i$ inside $E_i$, $d_i=S_i^2$. Take hermitian metrics on the 
line bundles, so that $B_c(S_1) \subset E_1$ and $B_c(S_2)\subset E_2$ are symplectic
manifolds for $c>0$ by using Lemma \ref{lem:2.1}.

For each $i=1,\ldots,m$, take small neighbourhoods $B(P_i)\subset S_1$,
symplectomorphic to the ball $B_c(0)$ via $f_{1i}:B(P_i) \to B_c(0)$. Take
a trivialization $\varphi_{1i}:E_1|_{B(P_i)} \stackrel{\cong}{\too} B(P_i) \x \CC$.
Therefore we have
 \begin{equation}\label{eqn:ot1}
  (f_{1i}\x \Id)\circ  \varphi_{1i}: B_c(S_1) \cap E_1|_{B(P_i)} \stackrel{\cong}{\too}  B_c(0) \x B_c(0)
  \end{equation}
Using Lemma \ref{lem:2.1}, we endow $E_1$ with a $2$-form $\o_{E_1}$ such that 
$(B_c(S_1),\o_{E_1})$ is symplectic  and the symplectic form is a product on 
$B_c(S_1) \cap E_1|_{B(P_i)}$. This means that (\ref{eqn:ot1}) is a symplectomorphism.
We do the same for $Q_i\in S_2$, obtaining a symplectomorphism
$f_{2i}:B(Q_i) \to B_c(0)$, a trivialization $\varphi_{2i}:E_2|_{B(Q_i)} \stackrel{\cong}{\too} B(Q_i) \x \CC$,
a symplectic form $\o_{E_2}$ on $B_c(S_2)$, and a symplectomorphism
$$
  (f_{2i}\x \Id)\circ  \varphi_{2i}: B_c(S_2) \cap E_2|_{B(Q_i)} \stackrel{\cong}{\too}  B_c(0) \x B_c(0)
$$
  
Let  $R:B_c(0)\x B_c(0)\to B_c(0)\x B_c(0)$, $R(z_1,z_2)=(z_2,z_1)$, be the map reversal of coordinates, which is a symplectomorphism
swapping horizontal and vertical directions. Then we take the gluing map
 $$
 \Phi_i= \left((f_{2i}\x \Id)\circ  \varphi_{2i}\right)^{-1}\circ R \circ 
 \left((f_{1i}\x \Id)\circ  \varphi_{1i}\right): B_c(S_1) \cap E_1|_{B(P_i)}  \to  B_c(S_2) \cap E_2|_{B(Q_i)}
 $$

\begin{definition}
We define the \emph{symplectic plumbing} $P_c(S_1\cup S_2)$ of $S=S_1\cup S_2$ as the
symplectic manifold
 $$
  X= (B_c(S_1)\sqcup B_c(S_2))/ x\sim \Phi_i(x), x \in B_c(S_1) \cap E_1|_{B(P_i)}, i=1,\ldots, m
  $$
\end{definition}

Note that $S_1\cup S_2\subset P_c(S_1\cup S_2)$ are symplectic submanifolds 
and they intersect transversely.

\subsection{Symplectic tubular neighbourhood} \label{subsec:tubular}

We need a symplectic tubular neighbourhood theorem for two intersecting surfaces
$S_1\cup S_2$. We start with the case of a single submanifold. We include the proof since
our result is a minor modification of the one appearing in the literature.

\begin{proposition}[Symplectic tubular neighborhood] \label{prop:tubular1}
Suppose that $(X,\o)$ and $(X',\o')$ are two symplectic $4$-manifolds (maybe open)
with compact symplectic surfaces $S \subset X$ and $S' \subset X'$.
Suppose that $S$ and $S'$ are symplectomorphic as symplectic manifolds via $f:S \to S'$,
and assume also that their normal bundles are smoothly isomorphic.

Let $V,V'$ be tubular neigbourhoods of $S$ and $S'$, with projections $\pi:V\to S$, $\pi':V'\to S'$,
and let $g:V\to V'$ be a diffeomorphism of tubular neighbourhoods of $S$ and $S'$ with $g|_S=f$.
Let $W\subset S$, $W'\subset S'$ such that $g|_{\pi^{-1}(W)}:\pi^{-1}(W)\to \pi'^{-1}(W')$ is a 
symplectomorphism. Suppose that $H^1(W)=0$, and let $\hat W \Subset W$.
Then there are tubular neighborhoods $S\subset U\subset X$ and
$S'\subset U'\subset X'$ which are symplectomorphic via $\varphi:U\to U'$, where $\varphi|_S=f$
and $\varphi|_{U\cap \pi^{-1}(\hat W)}=g$.
\end{proposition}

\begin{proof}
This is an extension of the symplectic tubular neighbourhood theorem  \cite{Silva2}, which is the case where
$W$ is empty. Let $g:V\to V'$ be the diffeomorphism of tubular neighbourhoods where $g|_S=f$. 
We start by isotopying $g$ so that $d_x g:T_xX \to T_{g(x)}X'$ is a linear symplectic map, for all $x\in S$.
We do this without modifying $g$ on $\pi^{-1}(\hat W)$, since $g$ is symplectic there. Then the symplectic orthogonal to $T_xS\subset T_xV$ is sent
to the symplectic orthogonal to $T_{f(x)}S'\subset T_{f(x)}V'$.

We take $\o_0=\o$ and $\o_1=g^*\o'$ and note that  $i^*(\o_1 - \o_0)=0$, where $i:S\to V$ is the 
inclusion map. As $i^*:H^2(V) \to H^2(S)$ is an isomorphism, we have that $[\o_1-\o_0]=0$,
hence there exists a $1$-form $\mu \in \Omega^1(V)$ such that $d \mu = \o_1 - \o_0$. 
We can suppose that $i^* \mu=0$, since otherwise 
we would consider the form $\mu - \pi^* i^* \mu$. 

Take an open set $\tilde W$ such that $\hat W\Subset \tilde W \Subset W$.
We can also suppose that $\mu|_{\pi^{-1}(\tilde W)}=0$. 
As $\o_1-\o_0=0$ on $\pi^{-1}(W)$,  $d\mu=0$ on $\pi^{-1}(W)$, and hence $\mu=df$ for some function $f\in C^\infty(\pi^{-1}(W))$, since 
we are assuming that $H^1(W)=0$. As $i^*\mu=0$ we can change $f$ by $f-\pi^*i^*f$, so that $df=\mu$ and $i^*f=0$.
Let  $\rho$ be a step function on $S$ such that $\rho|_{\tilde W}\equiv 1$ and $\rho \equiv 0$ outside $W$. Then
we can substitute $\mu$ by $\mu-d((\pi^*\rho) f)$. 

We can also suppose that the restriction $\mu|_S=0$.
In local coordinates $(x_1,x_2,y_1,y_2)$ where $S=\{(x_1,x_2,0,0)\}$, we have
$\mu= \sum a_j(x_1,x_2) dy_j + O(y)$.
We cover $S$ with balls $B_\a$, and then $(\mu|_S)|_{B_\a}=\sum a_j^\a dy_j^\a$.
The balls are chosen so that they are inside $S-\hat W$ or inside $\tilde W$. 
Take a partition of unity $\{\rho_\a\}$ subordinated to it. We define $k_\a =\sum a_j^\a y_j^\a$
and $k=\sum \rho_\a k_\a$. For those $B_\a\subset \tilde W$, we can take $k_\a=0$.  
Then $dk|_S= \mu|_S$, and we can substitute $\mu$ by $\mu-dk$. Note that
$k=0$ on $\pi^{-1}(\hat W)$, so we keep $\mu|_{\pi^{-1}(\hat W)}=0$.

Now consider the form $\o_t= t \o_1 + (1-t) \o_0= \o_0 + t \, d \mu$, for $0 \le t \le 1$. 
Since $d_xg$ is a symplectomorphism for all $x\in S$, we have $\o_1|_S=\o_0|_S$ and
hence $\o_t|_S=\o_0|_S$ is symplectic over all points of $S$. So, reducing $V$ if necessary,
$\o_t$ is symplectic on some neighborhood $V$ of $S$.
The equation $\iota_{X_t} \o_t= - \mu$ admits a unique solution $X_t$ which is a vector field on $V$.
By the above, $X_t|_S=0$ and $X_t|_{\hat W}=0$. Take the flow $\varphi_t$ of the family of vector fields $X_t$.
There is some $U\subset V$ such that $\varphi_t(U) \subset V$ for all $t \in [0,1]$.
Moreover $\varphi_0= \Id_{U}$, and $\varphi_t|_S= \Id_S$ and $\varphi_t|_{\hat W}=\Id_{\hat W}$. 
We compute
 \begin{align*}
 \frac{d }{dt}\Big|_{t=s} \varphi_t^* \o_t &= 
  \varphi_s^*\left({L}_{X_s} \o_s\right) +  \varphi_s^* (d \mu)  
 = \varphi_s^* \left(d(\iota_{X_s} \o_s) + \iota_{X_s} d \o_s \right) + \varphi_s^* d \mu  \\
  &= -\varphi_s^* (d\mu) + \varphi_s^* (d \mu) =0.
 \end{align*}
This implies that $\o_0= \varphi_0^* \o_0 = \varphi_1^* \o_1$.
So $\varphi_1:(U, \omega) \to (V,g^*\o')$ is a symplectomorphism.
The composition  $\varphi=g\circ \varphi_1: (U,\omega) \to (V',\o')$ is a symplectomorphism
of $U$ onto $U'=\varphi(U)\subset V'$.
\end{proof}

\subsection{Symplectic tubular neighbourhood of two intersecting submanifolds}

Now we move to the case of the union of two intersecting symplectic submanifolds.

\begin{definition}
Let $(X,\o)$ be a symplectic $4$-manifold.
We say that two symplecic surfaces $S_1,S_2\subset X$ intersect $\o$-orthogonally
if for every $\p\in S_1\cap S_2$ there are (complex) Darboux coordinates $(z_1,z_2)$ such
that $S_1=\{z_2=0\}$ and $S_2=\{z_1=0\}$ around $p$.
\end{definition}

By definition, $S_1$ and $S_2$ intersect $\o$-orthogonally in the symplectic plumbing $P_c(S_1\cup S_2)$.

\begin{lemma}[{\cite[Lemma 6]{MRT}}] \label{lem:symplectic-orthogonal}
Let $(X,\o)$ be a symplectic $4$-manifold, and suppose that $S_1, S_2 \subset X$ 
are symplectic surfaces intersecting transversely and positively. 
Then we can perturb $S_1$ to get another surface $S'_1$ in such a way that:
\begin{enumerate}
\item The perturbed surface $S'_1$ is symplectic.

\item The perturbation is small in the $C^0$-sense
and only changes $S_1$ near the intersection points with $S_2$, leaving
these points fixed, i.e.\ $S_1 \cap S_2= S'_1 \cap S_2$.

\item $S'_1$ and $S_2$ intersect $\o$-orthogonally.
\end{enumerate}
\end{lemma}

Let $S=S_1 \cup S_2 \subset X$ be a union of two intersecting symplectic submanifolds of
a symplectic manifold $X$.
We use the expression \emph{tubular neighborhood} of $S$ to refer to a small neighborhood $U$ of $S$
in $X$ such that $S$ is a deformation retract of $U$. 

\begin{theorem}[Symplectic tubular neighborhood] \label{thm:mine}
Suppose that $(X,\o)$ and $(X',\o')$ are two symplectic $4$-manifolds (maybe open)
with compact symplectic surfaces $S_1, S_2 \subset X$ and $S'_1, S'_2 \subset X'$.
Assume that $S_1$ and $S_2$ intersect symplectically orthogonally, and similarly for $S'_1$ and $S'_2$.
Suppose that there is a map $f:S=S_1\cup S_2 \to S'=S'_1\cup S'_2$ which is
a symplectomorphism $f:S_1\to S_1'$ and a symplectomorphism $f:S_2\to S_2'$.
Assume also that the normal bundles $\nu_{S_1}\cong \nu_{S_1'}$ and
$\nu_{S_2}\cong \nu_{S_2'}$. Then, there are tubular neighborhoods $S\subset U\subset X$ and
$S'\subset U'\subset X'$ which are symplectomorphic via $\varphi:U\to U'$, with $\varphi|_S=f$.
\end{theorem}

\begin{proof}
Take a point $P_i\in S_1\cap S_2$. Let $\varphi_i:B_i \to B_\e(0)\subset \CC^2$ be Darboux
coordinates so that $S_1=\{z_2=0\}$ and $S_2=\{z_1=0\}$, $\varphi_i(P_i)=0$.
For $f(P_i)\in S_1'\cap S'_2$ we also take $\varphi_i':B_i' \to B_\e(0)\subset \CC^2$ Darboux
coordinates so that $S_1'=\{z_2'=0\}$ and $S'_2=\{z_1'=0\}$, $\varphi_i'(f(P_i))=0$.
The composite $(\varphi_i')^{-1} \circ \varphi_i:B_i\to B_i'$ may not coincide with
$f$ on $B_i\cap (S_1\cup S_2)$. To arrange this, take 
 \begin{align*}
 &h_1=\varphi_i' \circ (f|_{B_i\cap S_1}) \circ \varphi_i^{-1}: B_{\e'}(0) \x \{0\} \to B_\e(0) \x \{0\} \\
 &h_2=\varphi_i' \circ (f|_{B_i\cap S_2}) \circ \varphi_i^{-1}: \{0\} \x B_{\e'}(0) \to \{0\} \x B_\e(0) 
 \end{align*}
which are symplectomorphisms onto their image. Then $h=h_1\x h_2$ is a symplectomorphism
of $\CC^2$ on a neighbourhood of the origin. So consider the symplectomorphism
  $$
  \psi_i=(\varphi_i')^{-1} \circ h \circ \varphi_i:W_i\to W_i'
   $$
defined on a neighbourhood $W_i\subset B_i$. It satisfies 
$\psi_i|_{B_i\cap (S_1\cup S_2)}= f|_{B_i\cap (S_1\cup S_2)}$. Fix also $\hat W_i\Subset W_i$,
and denote $W=\bigcup W_i$, $\hat W=\bigcup \hat W_i$, $W_i'=\psi_i(W_i)$, $W'=\bigcup W_i'$,
$\hat W_i'=\psi_i(\hat W_i)$, $\hat W'=\bigcup \hat W_i'$, and $\psi: W\to W'$
the map which is $\psi_i$ on each $W_i$.

Now take small tubular neighbourhoods $U_1,U_2$ of $S_1,S_2$ respectively. Then
$U_1\cap U_2$ is a neighbourhood of the intersection $S_1\cap S_2$, and can be made
as small as we want. We require that $U_1\cap U_2\subset \hat W$. We also
take neighbourhoods $U_1',U_2'$ of $S_1',S_2'$ respectively such that $U_1'\cap U_2'\subset
\hat W'$. We can define diffeomorphisms $g_j:U_j \to U_j'$ with $g_j|_{S_j}=f|_{S_j}$ and
$g_j|_{\tilde W \cap U_j}=\psi|_{\tilde W \cap U_j}$ for some $\hat W\Subset \tilde W \Subset W$, for $j=1,2$.
Apply Proposition \ref{prop:tubular1} to $g_j$, to obtain symplectomorphisms
 $\varphi_j: V_j  \to V_j'$, where $S_j\subset V_j \subset U_j$ and 
$S_j'\subset V_j' \subset U_j'$, such that $\varphi_j|_{S_j}=f|_{S_j}$ and $\varphi_j|_{\hat W\cap V_j}= 
\psi|_{\hat W \cap V_j}$. As $V_1\cap V_2 \subset U_1\cap U_2\subset \hat W$, we have that
 $\varphi_1,\varphi_2$ coincide in the overlap region, defining thus a symplectomorphism
  $$
   \varphi: V_1\cup V_2 \to V_1'\cup V_2'
   $$
 with $\varphi|_S=f|_S$.
\end{proof}

\begin{corollary} \label{cor:PcS}
Let $(X,\o)$ be a symplectic $4$-manifold and $S_1,S_2\subset X$ two compact symplectic surfaces intersecting
symplectically orthogonally. Then there is a neighbourhood $U$ of $S=S_1\cup S_2$ which is symplectomorphic
to a symplectic plumbing $P_c(S)$.
\end{corollary}

\begin{proof}
 Let $i_j: S_j \inc S$ be the inclusion map, and denote $\{P_1,\ldots, P_m\}= i_1^{-1}(S_1\cap S_2) \subset S_1$
 and $\{Q_1,\ldots, Q_m\}= i_2^{-1}(S_1\cap S_2) \subset S_2$. Take complex line 
 bundles $E_j\to S_j$ with $c_1(E_j)=d_j=[S_j]^2$, and define a symplectic plumbing $P_c(S_1\cup S_2)$ with
 these data. 
  Now apply Theorem \ref{thm:mine} to $S\subset X$ and $S\subset P_c(S)$.
 \end{proof}

 \begin{corollary} \label{cor:PcS2}
 Let $(S_1,\o_1), (S_2,\o_2)$ and $(S_1',\o'_1), (S'_2,\o'_2)$ be compact symplectic surfaces. Consider a symplectic 
 plumbing $P_c(S_1\cup S_2)$ with $\# S_1\cap S_2 =m$ and $d_j=[S_j]^2$, $j=1,2$, and another symplectic 
 plumbing $P_c(S'_1\cup S'_2)$ with $\# S'_1\cap S'_2 =m'$ and $d'_j=[S_j']^2$, $j=1,2$. If 
 $m=m'$, $\la [\o_j],[S_j]\ra =\la [\o'_j],[S'_j]\ra$ and $d_j=d_j'$, $j=1,2$, 
 then there are neighbourhoods $S_1\cup S_2 \subset U \subset P_c(S_1\cup S_2)$
 and $S_1'\cup S_2' \subset U' \subset P_c(S'_1\cup S'_2)$ which are symplectomorphic.
 \end{corollary}

\begin{proof}
Note that two compact surfaces $\S,\S'$ are symplectomorphic if and only if they have the same area
$\la [\o],[\S]\ra =\la [\o'],[\S']\ra$. Moreover the symplectomorphism can be chosen so
that it sends some finite collection of $m$ points of $\S$ to another collection of $m$ points
of $\S'$. Applying this to $S_j,S_j'$, we get a symplectomorphism $f_j:S_j \to S_j'$ with
$f_j|_{S_1\cap S_2}:S_1\cap S_2 \to S_1' \cap S_2'$ sending the intersection points
in the required order, $j=1,2$. Therefore $f_1|_{S_1\cap S_2}=f_2|_{S_1\cap S_2}$,
thus defining a map $f:S_1\cup S_2\to S_1'\cup S_2'$. As the intersections are symplectically
orthogonal, we can apply Theorem \ref{thm:mine} to get the stated result.
 \end{proof}

 This gives uniqueness of symplectic plumbings. In particular, they do not depend on the
 choices of symplectomorphisms of the surfaces, or the choice of Darboux coordinates at
 the intersection points.

\begin{remark}
Theorem \ref{thm:mine} holds for a symplectic manifold $X$ of
any dimension, and symplectic submanifolds $S_1,S_2\subset X$ of complementary
dimension intersecting symplectically  orthogonally.

The plumbing can be defined for symplectic manifolds $S_1,S_2$ of any
dimension $2n$, and $P_c(S_1\cup S_2)$ will have dimension $4n$. 
\end{remark}

\section{A configuration of symplectic surfaces in $\CP^2 \# 11 \overline{\CP}{}^2$.}\label{sec:3}

\subsection{Homology of $\CP^2 \# 11 \overline{\CP}^2$}
Let $X=\CP^2 \# 11 \overline{\CP}^2$ be the symplectic manifold obtained
by blowing up the projective plane $\CP^2$ at $11$ points $\D=\{P_1,\ldots, P_{11}\}$. We call $h \in H_2(X)$
the homology class of the line, and $e_i$, $1 \le i \le 11$, the homology classes of the
exceptional divisors, so that $H_2(X)= \langle h, e_1, \dots , e_{11} \rangle$.
Moreover, the intersection form of $X$ is diagonal with respect to the basis $\{h, e_1,\ldots, e_{11}\}$.
Now consider the collection of homology classes in $H_2(X)$ given by:
\begin{align*}
c_k &= 3h - \sum_{i \ne k}^{11} e_i, \qquad 1 \le k \le 11, \\
d & = 10 h - \sum_{i=1}^{11} 3 e_i
\end{align*}

\begin{proposition}
The homology classes $\{c_1, \dots , c_{11}, d\}$ form a basis of $H_2(X)$.
The intersection form is diagonal with respect to this basis, and the self-intersections are
$c_k^2=-1$, for $1 \le k \le 11$, and $d^2=1$.
\end{proposition}

\begin{proof}
It follows from the fact that $e_i \cdot h=0$ for all $i$, $e_i^2=-1$ and $h^2=1$.
This implies that the determinant of the intersection form with respect to this basis is $-1$, 
hence it is a basis over $\ZZ$.
\end{proof}

Our focus is to prove that the basis $\{c_1, \dots , c_{11}, d\}$ of $H_2(X)$ can be realized
by symplectic surfaces. For this, we need the following configuration
of symplectic surfaces in $\CP^2$:
\begin{itemize}
\item Eleven symplectic surfaces $C_1, \dots , C_{11}$ such that its homology classes are
$[C_i]=3h$ in $\CP^2$. These surfaces $C_i$, being cubics, must have $g=1$ 
by the symplectic adjunction formula. The surface $C_i$ is required to pass through
the $10$ points in $\D-\{P_i\}$, but not through $P_i$. Therefore, the proper transform
$\tilde C_i$ of $C_i$ in the blow-up $X=\CP^2 \# 11 \overline{\CP}{}^2$ of $\CP^2$ at 
$S$, has homology class $[\tilde C_i]=c_i$. 

\item The intersection  $C_i \cap C_j$ contains the $9$ points $\D- \{P_i,P_j\}$, for $i\neq j$.
Note that the algebraic intersection is $C_i\cdot C_j=9$.
If these intersections are transverse and positive (e.g.\ if $C_i$ are holomorphic around the intersection points)
and if there are no more intersections, then the proper transfroms $\tilde C_i, \tilde C_j$ are disjoint.

\item One singular symplectic surface $G$ such that $[G]=10h$ and $G$ has $11$
ordinary triple points at the points of $\D$. By the adjunction formula the genus of $G$ is
 $$
 g= \frac12 (10-1)(10-2) -11 \frac{3\cdot 2}2 = 36-33 =3.
 $$
If $G$ is holomorphic at a neighbourhood of the triple points, and the branches intersect
transversely (and hence also positively), then the proper transform $\tilde G$ of $G$ in the 
blow-up $X=\CP^2 \# 11 \overline{\CP}{}^2$ of $\CP^2$ at 
$S$, has homology class $[\tilde G]=10h - 3(e_1+\ldots + e_{11})=d$. 
Moreover, if there are no more singularities, then $\tilde G$ is a smooth symplectic surface in $X$.

\item The intersections $C_i \cap G$ contain the $10$ points $\D- \{P_i\}$.
Note that the algebraic intersection is $C_i\cdot G=30$.
If the intersections with each of the three branches at each intersection point
are transverse and positive (e.g.\ if $C_i, G$ are holomorphic around the intersection points),
and if there are no more intersections, then these account for all intersections.
In the blow-up $X$, the proper transfroms $\tilde C_i,\tilde G$ are disjoint.

\end{itemize}

Our aim now is to construct these surfaces in $\CP^2$. For this, we will make the 
construction in a local model, and then we will transplant it to $\CP^2$.

\subsection{Construction of a local model} 

Now we are going to construct the required $11$ surfaces of genus $1$ and the singular surface 
of genus $3$, in a local model.  The local model is as follows: take a genus $1$ complex curve
$C$ and a rational complex curve $L\cong \CP^1$. Take three points $Q_1,Q_2,Q_3\in C$ and
another three $Q_1',Q_2',Q_3'\in L$. Take a line bundle $E\to C$ of degree $9$ and
a line bundle $E'\to L$ of degree $1$, and perform the plumbing as given in Section \ref{sec:plumbing}.
This produces a symplectic manifold $P_c(C\cup L)$, which contains $C\cup L$.

\begin{proposition}
 Let $C'\subset \CP^2$ and $L'\subset \CP^2$ be a smooth cubic and a line in the complex plane, intersecting
 transversely. Then $P_c(C\cup L)$ can be symplectically embedded in a 
 neighbourhood of $C'\cup L'$, where $C$ is sent to $C'$ and $L$ is sent to a $C^0$-small perturbation of $L'$,
 preserving the intersection points.
 \end{proposition}
 
 \begin{proof}
 We start by modifying $L'$ to $L''$ using Lemma \ref{lem:symplectic-orthogonal}, 
 so that $C'$ and $L''$ intersect symplectically
 orthogonally. By Corollary \ref{cor:PcS}, a small neighbourhood of $C'\cup L''$ is symplectomorphic to
 a small neighbourhood of the plumbing of $C\cup L$, that is some $P_c(C\cup L)$, for $c>0$ small, and
 the symplectomorphism sends $C$ to $C'$ and $L$ to $L''$.
 \end{proof}

Therefore, to prove Theorem \ref{thm:main1}, it is enough to prove the following:

\begin{theorem}
There are $11$ points  $P_1,\ldots, P_{11}$ in $P_c(C\cup L)$, and symplectic surfaces 
 $C_1,C_2,\ldots, C_{11},G \subset P_c(C\cup L)$ such that: 
 \begin{itemize}
 \item $C_i$ is a section of a complex line bundle $E\to C$ of degree $3$, 
 and $P_j \in C_i$ for $j\neq i$, $P_i \notin C_i$. In particular, they have genus $1$.
 \item The surfaces $C_i, C_j$, $i \neq j$ intersect exactly at 
 $\{P_1, \ldots , P_{11} \}-\{P_i,P_j\}$, positively and transversely.
 \item $G$ is a genus $3$ singular symplectic surface whose only singularities
 are $11$ triple points at $P_i$ (with different branches intersecting positively).
 Moreover $G$ intersects each $C_i$ only at the points $P_j$, $j\neq i$, 
 and all the intersections of $C_i$ with the branches of $G$ are positive and transverse.
 \end{itemize}
\end{theorem}

To be more concrete, we do as follows. We fix a complex structure on $C$, and a degree $9$ 
complex line bundle $E\to C$. This is going to be as follows: take a 
complex disc $D \subset C$, which we assume as the radius $1$ disc $D=D(0,1)\subset \CC$.
Let $V=C-\bar D(0,1/2)$, and consider the change of trivialization given by the function $g(z)=z^9$
con $D(0,1)-\bar D(0,1/2)$. This means that $E$ is formed by gluing $E|_D=D\x \CC$ with $E|_V=V\x\CC$
via $(z,y) \sim (z, z^9 y)$. 
We endow $E$ with an auxiliary  hermitian metric which is of the form
$h(z)=1$ on the trivialization $E|_D$. 
We will choose the points $Q_1,Q_2,Q_3\in D \subset C$. We also take a complex line bundle $E' \to L$
of degree $1$, 
for which we fix a 
hermitian structure. 
Fixing three points $Q_1',Q_2',Q_3'\in L$, we perform the plumbing $P_c(C\cup L)$.

\subsection{Construction of the genus $1$ surfaces} \label{subsec:genus1}

The genus $1$ symplectic surfaces will be constructed as sections of the line bundle $E\to C$. 
Consider the previous cover $C=V \cup D$ and trivializations
$E|_V \cong V \x \CC$ and $E|_D \cong D \x \CC$.
Fix distinct numbers $z_1,\ldots, z_{10},z_{11} \in D$ with $z_{11}=0$ the origin.
Take $\l>0$ a small positive real number to be fixed later. Take the points
 \begin{equation}\label{eqn:pointsP}
 P_j=(\l z_j,0) , j=1,\ldots, 10, \text{ and } P_{11}=(0,1)
 \end{equation}
in $E$, in the given trivialization $E|_D=D\x \CC$. We define $11$ holomorphic local sections in the chart $D \subset C$ as
 $$
 \s_j(z)=\prod_{i\neq j}^{10} \left(1- \frac{z}{\l z_i} \right), \quad j=1,\ldots, 10,
 $$
and $\s_{11}(z)=0$. 

Clearly $\s_j(z)=\s_{11}(z)=0$ at the $9$ points $\l z_1, \ldots, \widehat{\l z_j}, \ldots, \l z_{10}$.
Also, for $1 \leq j < k \leq 10$, we have that $\s_j(z)=\s_k(z)$ at the $9$ points given by 
 \begin{equation}\label{eqn:9}
 \l z_1, \ldots, \widehat{\l z_j}, \ldots, \widehat{\l z_k}, \ldots, \l z_{10}, z_{11}=0 .
 \end{equation}
All the intersections of the graphs are transverse and positive since the points $\l z_i$ are simple roots and $\s_j$ are holomorphic sections.
By construction, the graph $\G(\s_j)$ of the local section $\s_j$ 
in the trivialization $E|_{D} \cong D \x \CC$ 
contains the set of points $\{P_1, \ldots, P_{11}\}-\{P_j\}$, as desired.

Now we move to the trivialization $E|_V$. Let us see that we can extend the sections $\s_j$
to all of $V$ without introducing any new intersection points between their graphs. 
For $z \in D \cap V$, the sections $\s_j$ become, for $|z| \geq 1/2$, in the trivialization of $E|_V \cong V\x \CC$,
 $$
 \tilde \s_j = z^{-9}\prod_{i\neq j}^{10} \left(1- \frac{z}{\l z_i} \right)
 = \l^{-9} A z_j \prod_{i\neq j}^{10} \left( 1-\frac{\l z_i}{z} \right), \quad A=- (z_1\ldots z_{10})^{-1},
 $$
and $\tilde \s_{11}=0$. Then $\tilde \s_j$ has the form 
 $$
 \tilde \s_j= A \lambda^{-9} z_j (1+ \l f_j(z,\l)),
 $$
where 
 $$
  f_j(z,\l)= \frac{1}{\lambda}\left( \prod_{i\neq j}^{10} \left( 1-\frac{\l z_i}{z} \right) -1\right)
  $$ 
is a holomorphic function of $z$ 
depending on the parameter $\l$, such that
$$
|f_j(z, \l)| \leq M_0 \text{ , for } \l \le \tfrac{1}{4}, |z| \ge \tfrac{1}{2},
$$
being $M_0$ a constant depending only on $z_1, \dots , z_{11}$.

Let $\rho$ be the smooth non-increasing function  with $\rho(r)=0$ for $r\geq 3/4$ 
and $\rho(r)=1$ for $r \leq 2/3$. Here $r=|z|$ is the radius in the disc $D$. 
Now we modify the local sections $\tilde \s_j$ to sections $\hat \s_j$ that
can be extended to global sections in $E \to C$.
We define for $z \in U$, $|z| \ge 1/2$,
\begin{align} \label{eqn:a1}
 \hat \s_j(z) & = \rho(|z|) \tilde\s_j(z)+ (1-\rho(|z|)) \lambda^{-9}A z_j = \lambda^{-9}Az_j \left(1+ \lambda \rho(|z|) f_j(z,\lambda)\right). 
\end{align}
We also put $\hat \s_{11}=0$.

We have that $\hat \s_j= \tilde \s_j$ in $\{1/2 \le |z| \le 2/3\}$, 
so $\hat \s_j$ extends to the trivialization $E|_{D}$ as $\s_j$ in $\{|z| \le 1/2\} \subset D$. 
Moreover, $\hat \s_j(z)= Az_j$ is constant for $|z| \ge 3/4$, 
so $\hat \s_j$ extends to all $V$, hence they give global sections in the line bundle $E \to C$.
We call $\hat \s_j$ these global sections, and $\G(\hat \s_j)$ their graphs.

Now let us check that no undesired intersection points are introduced between any pair of surfaces $C_j$, $1 \le j \le 11$.
On $|z| \le 1/2$, $\hat \s_j=\tilde\s_j$, so $\tilde\s_j$, $\tilde\s_k$, $j\neq k$, have $9$ intersection points 
given by (\ref{eqn:9}), which are the set $\{P_1, \ldots , P_{11}\} - \{P_j,P_k\}$. As $\tilde\s_j$ and $\tilde\s_k$
are holomorphic there, and the roots are simple, the intersections are positive and transverse.

For $|z|\geq 3/4$, $\hat\s_j =\lambda^{-9}Az_j$, $j=1,\ldots, 10$ and $\hat\s_{11}=0$. Therefore
the sections do not intersect since the $\{z_j\}$ are distinct points.
Now assume that
$1/2 \le |z| \le 3/4$. If $\hat \s_j(z) = \hat \s_k(z)$ with $k \ne j\leq 10$ then 
$$
 z_j+ z_j\lambda\rho(|z|) f_j(z,\l) = z_k + z_k\lambda\rho(|z|) f_k(z,\l).
$$
Taking $\l>0$ small enough, the discs $B(z_j,M_0|z_j| \l)$ and $B(z_k,M_0|z_k| \l)$ are all pairwise disjoint, 
so the above equality does not happen.
Analogously, if $\hat \s_j(z)=\hat\s_{11}(z)=0$ for $1/2 \le |z| \le 3/4$
we have a contradiction as long as $\l$ is small enough so that 
the discs $B(z_j,\l |z_j| M_0)$ do not contain the origin.

Finally  considering $\hat \s_j^\e= \e \hat \s_j$, being $\e>0$ small enough, the intersections 
of the graphs remains the same except that $P_{11}$ is changed to the point $(0,\e)$.
This ensures that the graphs are all contained in the given neighbourhood $B_c(C)$, for any $c>0$ given
beforehand. Moreover the graphs become $C^1$-close to the zero section 
$C \subset E$, 
in particular the graphs are symplectic surfaces of $B_c(E)$.

\subsection{Construction of the genus $3$ surface}

The genus $3$ surface will be constructed inside the neighbourhood $P_c(C\cup L)$ of
$C\cup L$, where $C$ is the genus $1$ surface and $L$ the genus $0$ surface,
both intersecting at three points. We will take $3$ sections of
the bundle $E\to C$, all of them passing through the eleven points
$P_1,\ldots, P_{11}$. In this way we get the $11$ triple points.
Then we add the line $L$, and glue the three sections with $L$
around the intersection points of $L$ and $C$. Let us carry out the details.

As before, take the previous cover $C=V \cup D$, $D=D(0,1)$, $V=C-\bar D(0,1/2)$, 
and trivializations $E|_D \cong D \x \CC$ and $L|_V \cong V \x \CC$,
with change of trivialization $g(z)=z^9$. We have fixed 
$z_1,\ldots, z_{10},z_{11}=0 \in D$ and the points
$$
 P_j=(\l z_j,0) , j=1,\ldots, 10, \text{ and } P_{11}=(0,1)
  $$
in $E|_D=D\x \CC$, where $0<\l \leq 1/4$ is some small number as arranged in Section \ref{subsec:genus1}.

We choose another three distinct values $w_1,w_2,w_3\in D$, different to $z_1,\ldots, z_{11}$.
We take the points 
\begin{equation} \label{eqn:points_Q}
 Q_1=(\l w_1,0), Q_2=(\l w_2,0), Q_3=(\l w_3,0),
\end{equation}
in the trivialization $E|_D=D\x \CC$. 
Consider (meromorphic) sections $\tau_k$, defined in $D - \{Q_k\}$ by the formula
$$
\tau_k(z)= \frac{\prod_{i=1}^{10} \left(1-\frac{z}{\l z_i}\right)} {1-\frac{z}{\l w_k}}\, ,
$$
for $k=1,2,3$. 
The graph of $\tau_k$ passes through all $11$ points (\ref{eqn:pointsP}). 

Let us see that we can extend the sections $\tau_k$ to the trivialization $E|_V$,
giving thus sections over $C - \{Q_k\}$. 
For $z \in D \cap V$, i.e.\ $|z| \ge 1/2$, we express $\tau_k$ in the trivialization $L|_V$, which is given by
 $\tilde \tau_k(z)=z^{-9} \tau_k(z)$. 
\begin{align*}
\tilde \tau_k(z) =z^{-9}  \frac{\prod_{i=1}^{10} \left(1-\frac{z}{\l z_i}\right)} {1-\frac{z}{\l w_k}} 
 = \lambda^{-9}A w_k \frac{\prod_{i=1}^{10} \left(1-\frac{\l z_i}{z}\right)} {1-\frac{\l w_k}{z}} = \lambda^{-9}Aw_k(1+ \lambda g_k(z,\l)),
\end{align*}
 where $A=-(z_1\cdots z_{10})^{-1}$ as before, and
 $$
  g_k(z,\l)= \frac{1}{\l} \left(\frac{\prod_{i=1}^{10} \left(1-\frac{\l z_i}{z}\right)} {1-\frac{\l w_k}{z}}-1\right)
  $$
are bounded functions for $|z| \ge 1/2$ and $0 < \l \le 1/4$,  say $|g_k(\l,z)| \le M$, for $M>0$ a constant.

Let $\rho:[0, \infty) \to \RR$ be a non-increasing smooth function 
such that $\rho(r)=1$ for $r \le 1/2$ and $\rho(r)=0$ 
for $r \ge 3/4$.
Now we modify $\tilde \tau_k(z)$ for $z \in D \cap V$, i.e.\ $|z| \ge 1/2$. 
Consider
$$
\hat \tau_k(z)= \lambda^{-9}Aw_k (1 + \rho(|z|) \l g_k(\l,z)).
$$
Clearly, $\hat \tau_k(z)=\tilde \tau_k(z)$ for $|z| \le 1/2$, so $\hat \tau_k$ extends to
the trivialization $E|_D$. Also, for $|z| \ge 3/4$ 
we have $\hat \tau_k(z)=\lambda^{-9}Aw_k $ is constant
so $\hat \tau_k$ extends to all the trivialization $L|_V$.
This yields a global section defined in $C - \{Q_k\}$, 
given by $\tau_k$ in $\{|z| \le 1/2\} \subset D$, and by $\hat \tau_k$ in $V$. 
We call from now on $\hat \tau_k$ this global section.
Let us denote 
 $$
 \Theta_k =\G(\hat\tau_k)= \{ (z,\hat \tau_k(z))\, | \, z \in C - \{Q_k\} \}
 $$ 
the graph of $\hat \tau_k$.

Let us see that the graphs $\Theta_1, \Theta_2, \Theta_3$
only intersect at the points $P_1, \dots, P_{10}, P_{11}$, i.e.\ that the sections only coincide
for the values $\l z_1, \ldots, \l z_{10}, z_{11}=0$. Let $j\neq k$.
On $|z|\leq 1/2$, $z \ne \l w_j, \l w_k$, if
$\hat \tau_j(z)=\hat \tau_k(z)$, then 
$$
\frac{\prod_{i=1}^{10} \left(1-\frac{z}{\l z_i}\right)} {1-\frac{z}{\l w_j}}=
\frac{\prod_{i=1}^{10} \left(1-\frac{z}{\l z_i}\right)} {1-\frac{z}{\l w_k}}.
$$
Hence either $z= \l z_i$ for some $1 \le i \le 10$ or $\frac{z}{\l w_j}=\frac{z}{\l w_k}$. 
The latter implies $z=0=z_{11}$.

For $|z| \ge 3/4$, if $\hat \tau_j(z)=\hat \tau_k(z)$ 
then $\lambda^{-9}Aw_j=\lambda^{-9}Aw_k$, which is false since $w_j\neq w_k$.
Finally, for $1/2 \le |z| \le 3/4$, if $\hat \tau_j(z)=\hat \tau_k(z)$ then 
$$
 w_j + w_j\rho(|z|) \l  g_j(\l,z))  = w_k+ w_k\rho(|z|) \l  g_k(\l,z)).
$$
Choosing $\l$ small enough, we have that the discs $D(w_j , \l |w_j|M)$ and $D(w_k , \l |w_k| M)$ do not intersect. 
So the above equality does not happen.

Finally, let us check the intersections of $\Theta_k$ with $\G(\hat\s_j)$.
Take $|z| \le 1/2$. Suppose that $\tau_k(z)=\s_j(z)$. This implies that 
$$
\s_j(z) \frac{1-\frac{z}{\l z_j}}{1-\frac{z}{\l w_k}} = \s_j(z),
$$
hence either $\s_j(z)=0$ or $\frac{z}{\l z_j} = \frac{z}{\l w_k}$.
In the first case we have that $z=\l z_i$ for some $i \ne j$.
In the second case we have that either $z=0=z_{11}$, 
or $\l z_j=\l w_k$ which is not possible because the points $w_k$ are different from the points $z_j$.

Suppose now that $1/2 \le |z| \le 3/4$ and $\s_j(z)=\tau_k(z)$. Then
$$
 z_j + z_j\rho(|z|) \l f_j(z,\l) =  w_k + w_k\rho(|z|) \l g_k(\l,z).
$$
if we take $\l$ small, the discs $D(z_{j}, M_0 |z_j| \l)$ and $D(w_k, M |w_k| \l)$
are disjoint, so the above equality is impossible. 
Finally, if $|z| \ge 3/4$ and $\tau_k(z)=\s_j(z)$ then $\lambda^{-9}Az_j=\lambda^{-9}Aw_k$, and this is false.

\subsection{Gluing the transversal in the plumbing}

The plumbing $P_c(C\cup L)$ is defined only for $c>0$ small enough. Let us arrange that our sections lie
inside it suitably. 
For this let $N>0$ be an upper bound of all $|\hat\s_j|$, $j=1,\ldots, 11$ such that $(|\hat\tau_k|)^{-1} ([N,\infty)) \subset
B(Q_k)$, where $B(Q_k) \subset D_{1/2}$ are small balls around $Q_k$, $k=1,2,3$. 
Recall that $\hat\tau_k=\tau_k$ is holomorphic on $B(Q_k)-\{Q_k\}$. 
As $\tau_k$ has a simple pole at $Q_k$, we have that 
 $$
 z'=z_k'=h_k(z)= \frac{1}{\tau_k(z)}
 $$
is a biholomorphism from a neighbourhood of $Q_k$ (that we keep calling $B(Q_k)$) to a ball $B_c(0)$. We take the
coordinate $z'_k$ on $B(Q_k)$. We need to modify the symplectic form so that $z'_k$ is also a Darboux coordinate.

\begin{lemma} \label{lem:juan}
 Consider the disc $D=D(0,1)$. We can perturb the standard symplectic form $\o_D$ to a nearby symplectic form $\o'_D$
 such that, maybe after reducing the balls $B(Q_k)$, the coordinate $z_k'$ are Darboux. The perturbation is 
 made only on a (slightly larger) ball around $Q_k$, and keeping the total area.
\end{lemma}

\begin{proof}
 We write $z'=z_k'=x'+iy'$. The standard symplectic form $\o_D$ on the coordinates $z$ is clearly K\"ahler, therefore
 it is also K\"ahler for the holomorphic coordinate $z'$. In particular, it has a K\"ahler potential $\phi(x',y')$,
 with $\o_D=\bd \bar\bd \phi(x',y')$. We can assume that $\phi$ has no linear part, 
 so $\phi(x',y')=\phi_2(x',y')+\phi_3(x',y')$, where $\phi_2(x',y')$
 is quadratic and $|\phi_3|=O(|(x',y')|^3)$. Then take some bump function $\rho$ that vanishes on a neighbourhood $B_\eta(Q_k)$
 of $Q_k$ (the size measured with respect to the radial coordinate $r'=|z'|$), 
 and $\rho\equiv 1$ on a slightly larger neighbourhood $B_{2\eta}(Q_k)$, $|d\rho|=O(\eta^{-1})$
 and $|\nabla d\rho|=O(\eta^{-2})$. Set 
  $\o_D'= \bd\bar\bd (\phi_2 + \rho\phi_3)$. Then $|\o'_D-\o_D|=O(\eta)$, and $\o'_D$ is standard on $(B_\eta(Q_k),z')$.
  As the difference $\o'_D-\o_D=\bd \bar\bd ((\rho-1)\phi_3)=d(\bar\bd (\rho-1)\phi_3)$ is exact and compactly supported,
  the total area remains the same.
\end{proof}

\begin{remark} 
Lemma  \ref{lem:juan} also holds in higher dimension. More concretely,
let $(Z,\o,J)$ be a K\"ahler manifold of real dimension $2n$, $p \in Z$
and $\f: U \to B \subset \CC^n$ holomorphic coordinates around $p$.
Then there exists a symplectic form $\o'$ on $Z$ so that $(Z,\o',J)$
is K\"ahler, $\o'=\o$ in $Z - U$, and $\o'$ is a linear symplectic form near $p$ on the coordinates $\f$,
in some smaller neighborhood $V\subset U$.
Moreover, the cohomology classes $[\o]=[\o']$. 
\end{remark}

\medskip

Over $E|_{B(Q_k)}\cong B_c(0)\x \CC$, the section $\tau_k$ is given by $v=1/z'$ (making $c>0$ smaller if needed), 
writing $z'=z_k'$ for brevity.
For $Q_1',Q_2',Q_3'\in L$, take  holomorphic balls $B(Q_j')\cong B_c(0)$, and arrange the
symplectic structure on $L$ to be standard over them. Finally, take symplectic structures on the total spaces
of the complex line bundles $\pi:E\to C$ and $\pi':E'\to L$ so that they are product symplectic structures
on $B_c(C)\cap E|_{B(Q_k)} \cong B(Q_k) \x B_c(0)$ and 
$B_c(L)\cap E'|_{B(Q'_k)} \cong B(Q'_k) \x B_c(0)$, respectively. The plumbing $P_c(C\cup L)$ is
done by gluing $B_c(C)$ and $B_c(L)$ along $R:B(Q_k) \x B_c(0) \to B(Q'_k) \x B_c(0)$, the map reversal of 
coordinates. Note that the uniqueness result of Corollary \ref{cor:PcS2} allows to do the plumbing with these choices.
We only have to take care of keeping the total areas $\la [C], [\o_{E}]\ra$ and $\la [L],[\o_{E'}]\ra$ fixed.

Now take $\e>0$ small enough so that:
 \begin{itemize}
 \item The graphs of the sections $\s_j^\e=\e \s_j$ are inside $B_c(C)$. For this $\e N<c$ is enough.
 \item The graphs of the sections $\s_j^\e$ are $C^1$-close to the zero section. This implies that these graphs are symplectic (a submanifold
 $C^1$-close to a symplectic one is symplectic).
 \item All sections $\tau_k^\e=\e\tau_k$ satisfy $|\tau_k^\e|<c$ on $C-B(Q_k)$, so the graph
 $\Theta_k^\e$ of $\tau_k^\e$ satisfies that $\Theta_k^\e \cap \pi^{-1}(C-B(Q_k)) \subset B_c(C)$. For this it is enough that $\e N<c$ again.
 \item The graphs of the sections $\tau_k^\e$ are $C^1$-close to the zero section on $C-B(Q_k)$, so they are symplectic.
 \end{itemize}

Now we look at the graph $\Theta_k^\e \cap (E|_{B(Q_k)})$. We have
 \begin{align*}
\Theta_k^\e \cap (E|_{B(Q_k)}) &\cong \left\{(z',v) \in B_c(0)\x \CC\, \Big|\,  v= \frac{\e}{z'}\right\} \\
  & =\left\{(z',v) \in B_c(0)\x \CC\, \Big| \, |v| \geq \e c^{-1}, z'= \frac{\e}{v}\right\} .
\end{align*}
Make $\e>0$ smaller if necessary, so that $\e c^{-1} \leq c/2$. Take $\rho(r)$ a smooth non-increasing
function so that $\rho(r)=1$ for $r\leq 1/2$ and $\rho(r)=0$  for $r\geq 3/4$. Define
 $$
 \hat\Theta_k^\e = \left\{(z',v) \in B_c(0)\x \CC\, \Big| \, \e c^{-1} \leq |v| \leq c,  z'= \e \rho(|v|/c) \frac{1}{v}\right\} 
 $$
This can be smoothly glued to  $\bar\Theta_k^\e =\Theta_k^\e \cap \pi^{-1}(C-B(Q_k))$. On the part
of the plumbing corresponding to $E'\to L$, this has the form 
$z'= \e \rho(|v|/c) \frac{1}{v}$ on $B(Q'_k)\x B_c(0)$, where $v$ is the coordinate for $B(Q_k')$ and
$z'$ is the vertical coordinate. Note that this can be extended as $z'=0$ in the bundle $E'\to L$,
over $L-(B(Q_1')\cup B(Q_2')\cup B(Q_3'))$.
The resulting smooth manifold is
  \begin{equation}\label{eqn:D}
  G=\bigcup_{k=1,2,3} \left(\bar\Theta_k^\e \cup \hat\Theta_k^\e\right) \cup \big(
 L-(B(Q_1')\cup B(Q_2')\cup B(Q_3'))\big).
  \end{equation}

Clearly, as $|v|\geq \e c^{-1}$ for the points of $\hat\Theta_k^\e$, there are no new intersections
with the graphs $\G(\s_j^\e)$ or $\Theta^\e_l$, $l\neq k$, since they are bounded by $\e N$, and 
we can take $c<N^{-1}$ to start with.

The graphs $\hat\Theta_k^\e$ are symplectic, since the graphs of
$z'= \e \rho(|v|/c) \frac{1}{v}$ are symplectic over $|v|\geq c/2$, by taking $\e>0$ small enough
so that it is $C^1$-close to the zero section $z'=0$ of the bundle $E'\to L$. On $\e c^{-1}\leq |v|\leq c/2$, the graph
coincides with $z'= \e \frac{1}{v}$, which is holomorphic hence symplectic.

\begin{remark}
The homology class of the graph $\G(\hat\s_j)$ in $P_c(C\cup L)$ is equal to $[C]$, since they are sections
of $E\to C$. The manifold $G$ of (\ref{eqn:D}) can be retracted to $3[C]+[L]$ in $P_c(C\cup L)$, by making $\e\to 0$.

When we embed $P_c(C\cup L)\hookrightarrow \CP^2$, the class $[C]\mapsto 3h$, and $[L]\mapsto h$,
where $h$ is the class of the line in $\CP^2$. Hence $[G]\mapsto 10h$, so $G$ has degree $10$.

The genus of $G$ is $3$ since topologically it is the gluing (connected sum) of three punctured surfaces of genus $1$ 
(given by the graphs of the sections 
$\Theta_k$, $k=1,2,3$), with a sphere with $3$ holes given by $L-(B(Q_1')\cup B(Q_2')\cup B(Q_3'))$.
\end{remark}

\section{Fundamental group of the K-contact $5$-manifold} \label{sec:4}

Let $X$ be the symplectic manifold constructed as the symplectic blow-up
of $\CP^2$ at the eleven points $P_1,\ldots, P_{11}$. 
The underlying smooth manifold is $X=\CP^2 \# 11\overline{\CP^2}$, with $b_2=12$.
It has $11$ surfaces of genus $1$, named $\tilde C_1,\ldots, \tilde C_{11}$, 
and a genus $3$ surface $\tilde G$ all of them disjoint.
We set the isotropy of $\tilde C_i$ to be $\ZZ/(p^i)$, $i=1,\ldots, 11$, and that of $\tilde G$ to be $\ZZ/(p^{12})$, 
for a fixed prime $p$. This determines a symplectic orbifold $X'$ uniquely by \cite[Proposition 7]{MRT}.

We start by computing the \emph{orbifold fundamental group} $\pi_1^{\orb}(X')$ of $X'$.
The reader can find alternative definitions  in \cite[Chapter 13]{Thu} and in \cite[Definition 4.3.6]{BG}.
We only need a presentation of $\pi_1^{\orb}(X')$, which follows from \cite[Th\'eor\`eme A.1.4]{HD}. For this,
fix a base point $p_0\in X'$. Take loops from $p_0$ to a point near $\tilde C_i$, followed by a loop
$\delta_i$ around $\tilde C_i$, and going back to $p_0$, $i=1,\ldots, 11$. In the same vein,
we add another loop $\delta_{12}$ around $\tilde G$. Then 
 $$ 
 \pi_1^{\orb}(X')= \frac{\pi_1(X-(\tilde C_1\cup \ldots \cup \tilde C_{11}\cup \tilde G))}{\la 
 \delta_1^p,\ldots, \delta_{11}^{p^{11}}, \delta_{12}^{p^{12}}\ra}\, .
 $$
Let us see that $\pi_1^{\orb}(X')$ is trivial.
It suffices to see that $\pi_1(X-(\tilde C_1\cup \ldots \cup \tilde C_{11}\cup \tilde G))$
is trivial. We start with a lemma.
 
\begin{lemma} \label{lem:loops-cubic}
We can arrange a complex cubic curve and a complex line $C', L' \subset \CP^2$
intersecting transversally, such that a small neighborhood $B_\epsilon(C' \cup L')$  of $C' \cup L'$
satisfies the following: there are generators of $\pi_1(C')$ represented by loops $\a,\b$
away from $B_\epsilon(L')$, that can be homotoped (outside
$B_\epsilon(L')$) to loops $\hat\a,\hat\b$ in $\bd B_\epsilon(C')$. 
The loops $\hat\a,\hat\b$ are contractible in
$\CP^2 - (B_\epsilon(C' \cup L'))$.
\end{lemma}

\begin{proof}
We consider a particular family of complex cubics in $\CP^2$ given by the affine equations
$C_r=\{ y^2=x^3-r^2 x\}$, with $r>0$ small. As $r \to 0$, the cubic $C_r$ collapses to
a cuspidal rational curve $C_0=\{y^2=x^3\}$, which has trivial first homology group. 
It is known \cite{Le} that the vanishing cycles generate the homology $H_1(C_r)$.
Here we give an explicit description, as the loops 
 \begin{align*}
  \a_r&=\{(x,y)| x\in  [-r,0], y\in \RR, y^2=x^3-r^2 x\}, \\
  \b_r &=\{(x,y)| x=-x'\in [0,r], y=iy' \in i\RR, (y')^2 = (x')^3-r^2 x'\}.
  \end{align*}
Note that $\a_r,\b_r$ intersect transversally at one point, hence they generate $\pi_1(C_r)\cong \ZZ^2$.
The homotopies given by $\a_t$, $t\in [0,r]$, and $\b_t$, $t\in [0,r]$ (with base-point at $(0,0)$) 
produce discs that contract $\a_r,\b_r$. These discs do not intersect $C_r$. Now fix some $C'=C_r$ and
take a tubular neighbourhood $B_\e(C')$ by considering all $C_s$ with $|s-r|<\e$. Then we can homotop
the loops $\a=\a_r,\b=\b_r$ to $\hat\a=\a_{r-\e},\hat\b=\b_{r-\e}$ which lie at the boundary,
and can be contracted outside $B_\epsilon(C')$.

Finally, take a complex line $L'\subset \CP^2$ intersecting transversally $C'$, but well away from
the loops $\a_r,\b_r$ and the homotopies above (e.g.\ a small perturbation of the line at infinity).
Therefore all previous statement happen outside $B_\e(L')$.
\end{proof}

\begin{proposition}
We have that the fundamental group  
$\pi_1(X-(\tilde C_1\cup \ldots \cup \tilde C_{11}\cup \tilde G))=1$. In particular, $\pi_1^{\orb}(X')=1$.
\end{proposition}

\begin{proof}
We constructed $C_1,\ldots, C_{11},G$ inside a plumbing $\mathbf{P}=P_c(C\cup L)$, and then we
have transferred it to a neighbourhood $\mathbf{P}'=P_c(C'\cup L'')$ of
a cubic $C'$ and a perturbation $L''$ of a line $L'$ in $\CP^2$. Note that $C'\cup L''$ is smoothly
isotopic to $C'\cup L'$. Then we blew-up at the eleven points $P_1,\ldots, P_{11}$ which lie
inside $\mathbf{P}'$, and took the proper transforms $\tilde C_1,\ldots, \tilde C_{11},\tilde G \subset \tilde{\mathbf{P}}{}'$,
where $\tilde{\mathbf{P}}{}'$ is the blow-up of $\mathbf{P}'$. Let $B_\epsilon(\tilde C_i), B_\epsilon(\tilde G)\subset \tilde{\mathbf{P}}{}'$ be small and
disjoint tubular neighbourhoods of $\tilde C_i, \tilde G$, $i=1,\ldots, 11$, respectively.

Put $X=W \cup W'$, with 
$$
W=\bigcup_i B_{2\epsilon}(\tilde C_i) \cup B_{2\epsilon}(\tilde G) \cup T_0 \,\, , \quad 
W'= X - \big(\bigcup_i B_{\epsilon}(\tilde C_i) \cup B_{\epsilon}(\tilde G) \big)
$$
where $T_0$ denotes an open contractible set constructed by fattening paths joining the base point
with the tubular neighbourhoods $B_{2 \epsilon}(\tilde C_i), B_{2 \epsilon}(\tilde G)$.
As $\pi_1(X)$ is trivial, Seifert Van-Kampen theorem shows that the map
 $$
 \pi_1(W \cap W') \longrightarrow \pi_1(W') \cong \pi_1(X-(\tilde C_1\cup \ldots \cup \tilde C_{11}\cup \tilde G))
 $$
is surjective. Note that $W \cap W'$ is homotopy equivalent to the wedge sum
$Y_1 \vee \dots \vee Y_{11} \vee Y_{12}$, 
where $Y_i=\bd B_\epsilon(\tilde C_i)$ is the boundary of a small tubular neighbourhood of $C_i$, and $Y_{12}=\bd B_\epsilon(\tilde G)$.
Hence it is enough to see that every loop in $Y_i$ for $1 \le i \le 11$
and every loop in $Y_{12}$ are contractible in $\pi_1(X-(\tilde C_1\cup \ldots \cup \tilde C_{11}\cup \tilde G))$. 

Take the plumbing $\mathbf{P}=P_c(C\cup L)$ and the curves $C_1,\ldots, C_{11},G$. 
We have decomposed $C=D\cup V$, where $D$ is a disc, so we may take $\a,\b$ inside $C-D$. 
For each of the cubics $C_i$, $\pi_1(C_i)$ is generated by loops $\a_i,\b_i$ which can
be taken by lifting $\a,\b$ via the sections $\hat\s_i$, $i=1,\ldots, 11$, of the complex line bundle $E\to C$. 
For the curve $G\subset \mathbf{P}$ of genus $3$, we have generators 
$\a^{(1)},\b^{(1)},\a^{(2)},\b^{(2)},\a^{(3)},\b^{(3)}$ of
the fundamental group $\pi_1(G)$ with $\prod_{j=1}^3 [\a^{(j)},\b^{(j)}]=1$. These can be
taken by lifting the loops $\a,\b$ via the sections $\hat\tau_j$, $j=1,2,3$.
The base point is also chosen outside the disc $D$.
In  $\mathbf{P}-(C_1\cup \ldots C_{11}\cup G)$, we can move vertically (along the fiberwise
directions of the bundle $E\to C$) all the loops $\a_i,\b_i, \a^{(j)}, \b^{(j)}$ without
touching the other curves. Once we reach the boundary of $\mathbf{P}\cong\mathbf{P}'$, these can be contracted in the complement
$\CP^2-\mathbf{P}'$ by Lemma \ref{lem:loops-cubic} above.

Now we blow-up inside $\mathbf{P}$ the eleven points $P_1,\ldots, P_{11}$ to obtain
$\tilde{\mathbf{P}}$ and the proper transforms $\tilde C_1,\ldots, \tilde C_{11},\tilde G$. 
Consider $Y_i=\bd B_\epsilon(\tilde C_i)$ as before. This is a circle bundle $S^1 \to Y_i 
=\bd B_\epsilon(\tilde C_i) \to \tilde C_i$
with Chern class $c_1(Y_i)=[\tilde C_i]^2=-1$. 
We have a short exact sequence 
$$
0 \to \pi_1(S^1) \to \pi_1(Y_i) \to \pi_1(\tilde C_i) \to 0.
$$
Since we are away from the blow-up locus we call the generators of $\pi_1(\tilde C_i)$ again $\a_i, \b_i$.
The loop $[\a_i,\b_i]$ can be homotoped in $B_\epsilon(\tilde C_i)$
to the base point through a homotopy transversal to $\tilde C_i$. This homotopy
intersects $\tilde C_i$ in $\tilde C_i^2=-1$ points counted with signs. 
Via the retraction
$B_e(\tilde C_i)-\tilde C_i \to Y_i$, 
this gives a homotopy in $Y_i$ between the lifting of $[\a_i,\b_i]$ and
$\g_i^{-1}$, where $\g_i$ is the loop going along the fiber $S^1$.
We conclude that
  $$
   \pi_1(Y_i)= \la \a_i,\b_i,\g_i\,  | \, [\a_i,\b_i]=\g_i^{-1} ,\g_i \text{ central}\ra.
   $$
Note that $\a_i,\b_i$ can be moved to $Y_i$ without touching the other cubics $\tilde C_j$ and then contracted in $\CP^2-\mathbf{P}$ via the blow-up map. 
The conclusion is that $\a_i$ and $\b_i$ can be contracted to a point through a homotopy in $X-(\tilde C_1\cup \ldots \cup \tilde C_{11}\cup \tilde G)$.
Therefore the same happens to $\g_i$.

Analogously, $Y_{12}=\bd B_\epsilon(\tilde G)$
 is a circle bundle $S^1\to Y_{12}=\bd B_\epsilon(\tilde G) \to \tilde G$
with Chern class $c_1(Y_{12})=[\tilde G]^2=1$. Denoting by $\g_{12}$ the loop along the fiber $S^1$, we have that
  $$
   \pi_1(Y_{12})= \la \a^{(1)},\b^{(1)},\a^{(2)},\b^{(2)},\a^{(3)},\b^{(3)}, \g_{12}| \prod_{j=1}^3 [\a^{(j)},\b^{(j)}]=\g_{12},
    \g_{12} \text{ central}\ra.
   $$
The loops $\a^{(j)},\b^{(j)}$ can be moved to the boundary $Y_{12}$ and then contracted in $\CP^2-\mathbf{P}$ via the blow-up map. Thus 
the same happens to $\g_{12}$. So all generators of $\pi_1(\bd B_\epsilon(\tilde C_i))$, $i=1,\ldots,11$, and of $\pi_1(\bd B_\epsilon(\tilde G))$, become
trivial in $\pi_1(X-(\tilde C_1\cup \ldots \cup \tilde C_{11}\cup \tilde G))$.
This concludes the proof.
\end{proof}

Once we have the symplectic orbifold $X'$, we construct
a Seifert bundle $M\to X'$ with primitive Chern class $c_1(M/e^{2\pi i/\mu})=[\omega]$.
This is a K-contact manifold, which is simply-connected.

\begin{theorem} 
The $5$-manifold $M$ is simply-connected, hence it is a Smale-Barden manifold.
\end{theorem}

\begin{proof}
 By \cite[Theorem 4.3.18]{BG}, we have an exact sequence $\pi_1(S^1)=\ZZ \to \pi_1(M) \to \pi_1^{\orb}(X')=1$.
 In particular, $\pi_1(M)$ is abelian. Therefore $\pi_1(M)=H_1(M,\ZZ)=0$, by Theorem \ref{thm:12}.
\end{proof}

\section{Non-existence of an algebraic surface with the given pattern of curves} \label{sec:5}

In this Section we show that it is not possible to construct an algebraic surface with the same 
configuration of complex curves as the manifold we constructed in Section \ref{sec:3}, that is 
$12$ disjoint complex curves spanning $H_2(S,\QQ)$, one of genus $3$ and all the others 
of genus $1$. More concretely, we prove Theorem \ref{thm:conj2}. 

\begin{theorem}\label{thm:a-g}
Suppose $S$ is a complex surface with $b_1=0$ and disjoint smooth complex curves 
spanning $H_2(S,\QQ)$, one of them of genus $g\geq 1$ and all the others elliptic (thus of genus $1$). 
Then $b_2\leq2g^2-4g+3$.
\end{theorem}

\begin{proof}
Let $S$ be a complex surface with $b_1=0$, containing 
disjoint complex curves spanning $H_2(S,\QQ)$, one of them, say $D_1$, of genus $g$ 
and the other curves $D_2, \ldots, D_{b_2}$ all of genus $1$.

As $\set{D_1,\ldots, D_{b_2}}$ is a basis of $H_2(S,\QQ)$, the Poincar\'e duals $[D_1],\ldots ,[D_{b_2}]$ 
are a basis of $H^2(S,\QQ)$. Furthermore, these classes are all of type $(1,1)$, so we 
have that $h^{1,1}=b_2$ and the geometric genus is $p_g=h^{2,0}=0$. The irregularity is $q=h^{1,0}=0$ since $b_1=0$.
In particular, $S$ is an algebraic surface \cite{BPV}. The holomorphic Euler characteristic is 
 \begin{equation}\label{eqn:chios}
 \chi(\cO_S)=1-q+p_g=1.
 \end{equation}
By the Riemann-Hodge relations, the signature of $H^{1,1}(S)$ is $(1,b_2-1)$. Therefore, 
the self-intersection of one of the $D_i$'s is positive and it is negative for the others.

\noindent {\bf Case 1.} Assume for the moment that $g=g(D_1)\geq 2$. 
We show first that $D_1^2>0$. Suppose otherwise that $D_i^2>0$ for one of the genus $1$ curves. 
After reordering, we can suppose this is true for $D_2$.
By the adjunction formula, $K_S\cdot D_2+D_2^2=2g(D_2)-2=0$, so $K_S\cdot D_2=-D_2^2$. 
And, by Riemann-Roch, we have,
 $$
 \chi(D_2) = \chi(\cO_S) + \frac{D_2^2 - K_S\cdot D_2}{2} = 1 + D_2^2.
 $$
Hence using Serre duality,
 $$
  h^0(D_2)+h^0(K_S-D_2)=h^0(D_2)+ h^2(D_2) \geq \chi(D_2)= 1+D_2^2\geq2.
  $$ 
Also, from the exact short sequence 
$0\to\cO_S(K_S-D_2)\to\cO_S(K_S)\to\cO_{D_2}(K_S|_{D_2})\to0$ we deduce that 
$h^0(K_S-D_2)=0$, since $h^0(K_S)=h^{2,0}(S)=0$. Thus $h^0(D_2)\geq 2$ and 
we can consider a pencil $\PP^1\leq|D_2|$. This gives a rational map $S\dashrightarrow\CP^1$ 
and, after blowing-up the base points of the pencil, an elliptic fibration $\tilde S\to \CP^1$, 
with the proper transform of $D_2$ as a smooth fiber. However, since they are pairwise disjoint, all the 
proper transforms of $D_i$, $i\neq 2$ have to lie in fibers. To see this, 
consider the projections of the proper transforms of $D_i$ by the elliptic fibration $\tilde S\to \CP^1$.
These projections must be Zariski-closed, connected subsets of $\CP^1$. As they are not all
$\CP^1$ since they do not intersect the proper transform of $D_2$, they should be points. 
In particular, since all the fibers are 
connected, the arithmetic genus of each irreducible component of a fiber has to be at most $1$, which 
gives rise to a contradiction as $g(D_1)=g>1$. Therefore, $D_1^2>0$ and $D_2^2,\ldots, D_{b_2}^2<0$.

Denote $m_1=D_1^2$ and $m_i=-D_i^2$, $i=2,\ldots, b_2$. All of the $m_i$'s are positive integers. 
And write $K_S\equiv\sum_{i=1}^{b_2}\lambda_iD_i$ its homology class in $H_2(S,\QQ)$, with $\lambda_i\in\QQ$. 
Notice that $K_S\cdot D_i=\lambda_iD_i^2$, from where $\lambda_i=\frac{K_S\cdot D_i}{D_i^2}$. And, by the adjunction formula,
\begin{align} \label{eqn:JR}
 &K_S\cdot D_1=2g(D_1)-2-D_1^2=2g-2-m_1, \nonumber \\
 &K_S\cdot D_i=2g(D_i)-2-D_i^2=m_i, \quad i\geq2.
\end{align}
Therefore
 $$
  K_S\equiv\frac{2g-2-m_1}{m_1}D_1-\sum\limits_{i=2}^{b_2}D_i\, ,
  $$
and we get
 \begin{equation}\label{eqn:o1}
  K_S^2=\frac{(2g-2-m_1)^2}{m_1}-\sum_{i=2}^{b_2}m_i\, .
 \end{equation}

Consider the following short exact sequence of sheaves,
 $$
  0\too \cO(K_S) \too\cO(K_S+D_1)\too \cO_{D_1}(K_{D_1}) \too0,
  $$
where $K_{D_1}=(K_S+D_1)|_{D_1}$, by adjuction. This gives a long exact sequence in cohomology,
 $$
 0\too H^0(K_S)\too H^0(K_S+D_1)\too H^0(K_{D_1})\too H^1(K_S)\too\ldots
 $$
where $H^0(K_S)=H^{2,0}(S)=0$ and $H^1(K_S)=H^1(\mathcal{O}_S)=H^{0,1}(S)=0$. So we have 
an isomorphism $H^0(K_S+D_1) \cong H^0(K_{D_1})$, and we deduce that $h^0(K_S+D_1)=h^0(K_{D_1})=g$. 
In particular, the linear system $|K_S+D_1|$ is not empty, and it has dimension $g-1\geq 1$. 
Let $Z=Z(|K_S+D_1|)$ be the fixed part of $|K_S+D_1|$ (that is, the largest effective divisor such that $D\geq Z$
for all $D\in |K_S+D_1|$). 
Notice that $Z\cdot D_1=0$, since the restriction of the linear system to $D_1$, $|(K_S+D_1)|_{D_1}|=|K_{D_1}|$, 
has no fixed points, as $g\geq 2$.

Write now $Z$ as an effective divisor $Z=\sum_{i=1}^{b_2}\alpha_iD_i+T$ where $\alpha_i$ 
are non-negative integers and $T$ is an effective divisor not containing any of the $D_i$'s. 
Notice that the latter implies $T\cdot D_i\geq0$, for all $i$.
Since $0=Z\cdot D_1=\alpha_1m_1+T\cdot D_1$, we have $\alpha_1=0$ and $T\cdot D_1=0$. 
So we can write $Z=\sum_{i=2}^{b_2}\alpha_iD_i+T$, and $T$ does not intersect $D_1$.

Let us see that $T=0$. Write $T\equiv\sum_{i=1}^{b_2}\mu_iD_i$ its homology class in $H_2(S,\QQ)$, 
with $\mu_i\in\QQ$. First note that, since $T\cdot D_1=0$, it is $\mu_1=0$, so $T\equiv\sum_{i=2}^{b_2}\mu_iD_i$.  
For $i\geq2$, $0\leq T\cdot D_i=-\mu_im_i$, hence $\mu_i\leq 0$. 
Let $n\geq1$ be an integer such that $n\mu_i\in\ZZ$ for all $i$.  Hence $nT$ is efective and $-nT= \sum (-n\mu_i)D_i$ is
also effective. This implies that $nT=0$ and thus $T=0$. This means that the fixed part is $Z=\sum_{i=2}^{b_2}\alpha_iD_i$.

Write $|K_S+D_1|= Z+ |F|$, where $F$ is the free part, which is a fully movable divisor. We look now
at the self-intersection $F^2=(K_S+D_1-Z)^2\geq 0$. Recall that the self-intersection of a fully movable divisor is $F^2\geq 0$,
since taking a different $F'\equiv F$, such that $F$ and $F'$ do not share components, then $F^2=F\cdot F'\geq 0$. 

Let $j\geq2$ and suppose both that $m_j=1$ and $D_j\nleq Z$. In this case, the restriction of an effective divisor
$C\in|K_S+D_1|$ to $D_j$ is an effective degree $1$ divisor, since $(K_S+D_1)\cdot D_j= K_S\cdot D_j=m_j=1$, by (\ref{eqn:JR}).
So $C\cap D_j$ is a point $P_j$. Furthermore, 
since $D_j$ is not a rational curve, any pair of linearly equivalent points are actually equal. Therefore 
$P_j\in D_j$ is a fixed point of $|K_S+D_1|$ and, since $Z\cap D_j=\emptyset$, $P_j\in F$. So $P_j$ 
is counted in the self-intersection $(K_S+D_1-Z)^2$.

There are at most $\sum_{i=2}^{b_2}m_i-(b_2-1)$ curves among $D_2,\ldots, D_{b_2}$ with $m_i>1$. 
So there are at least $2(b_2-1)-\sum_{i=2}^{b_2}m_i$ curves with self-intersection $-1$. 
Hence there are at least $2(b_2-1)-\sum_{i=2}^{b_2}m_i-r$ fixed points $P_j\in D_j$ of some
$|(K_S+D_1-Z)|_{D_j} |$, where $r=\#\set{\alpha_i>0}$, and thus
 \begin{equation}\label{eqn:o2}
  (K_S+D_1-Z)^2\geq2(b_2-1)-\sum_{i=2}^{b_2}m_i-r.
 \end{equation}
Note that in case we have $2(b_2-1)-\sum_{1=2}^{b_2}m_i-r\leq0$ we cannot assure the 
existence of any fixed point but the inequality still holds since $(K_S+D_1-Z)^2\geq0$.

We now compute $(K_S+D_1-Z)^2$,
 \begin{align*}
  (K_S+D_1)^2 &=K_S^2+2K_S\cdot D_1+D_1^2=K_S^2+2(2g-2-m_1)+m_1 = \\ 
  &=K_S^2+4g-4-m_1\, , \\
 (K_S+D_1-Z)^2&=(K_S+D_1)^2-2(K_S+D_1)\cdot Z+Z^2 = \\ 
  &=K_S^2+4g-4-m_1-2\sum_{i=2}^{b_2}\alpha_im_i-\sum_{i=2}^{b_2}\alpha_i^2m_i\, .
  \end{align*}
Thus (\ref{eqn:o2}) gives 
 $$
 K_S^2+4g-4-m_1-2\sum_{i=2}^{b_2}\alpha_im_i-\sum_{i=2}^{b_2}\alpha_i^2m_i\geq2(b_2-1)-\sum_{i=2}^{b_2}m_i-r\, ,
 $$
from where
 \begin{align*}
 2b_2 &\leq 4g-2-m_1+K_S^2+\sum_{i=2}^{b_2}m_i+r-2\sum_{i=2}^{b_2}\alpha_im_i-\sum_{i=2}^{b_2}\alpha_i^2m_i \\
 &\leq 4g-2-m_1+K_S^2+\sum_{i=2}^{b_2}m_i+r-3r\leq4g-2-m_1+K_S^2+\sum_{i=2}^{b_2}m_i \, .
  \end{align*}
Using (\ref{eqn:o1}), we have
 $$
  2b_2\leq4g-2-m_1+\frac{(2g-2-m_1)^2}{m_1}\, .
  $$
The expression on the right is a decreasing function on $m_1$. Therefore, we can bound it by its value in $m_1=1$, that is
 $$
 2b_2\leq 4g-3+(2g-3)^2=4g^2-8g+6.
 $$
Hence $b_2\leq 2g^2-4g+3$, as required.

\noindent \textbf{Case 2.} Suppose now $g=1$. Using the adjunction formula, we get
$K_S\equiv-\sum_{i=1}^{b_2}D_i$.
And using the same argument as above, we have that $h^0(K_S+D_1)=h^0(K_{D_1})=g=1$. 
Thus, there is an effective divisor in $S$ linearly equivalent to $K_S+D_1=-\sum\limits_{i=2}^{b_2}D_i$ 
which is clearly anti-effective if $b_2\geq2$. Therefore, $b_2\leq1$.
\end{proof}

\bigskip

Let us end up by giving a different proof of the non-existence of a Kähler surface $S$ with $b_1=0$ and $b_2=12$,
containing disjoint smooth complex curves 
spanning $H_2(S,\QQ)$, one of them of genus $g=3$, all the others of genus $g_i=1$. It makes very specific use
of the numbers at hand.

We follow the notations in the proof of Theorem \ref{thm:a-g}. We have the curves $D_1,D_2,\ldots, D_b$, $b=12$,
with $D_1^2=m_1$, $D_i^2=-m_i$, $2\leq i\leq b$, and all $m_i$'s are positive integers. The curve $D_1$ has
genus $g=3$ and $D_i$ have genus $1$, $2\leq i\leq b$. 
By (\ref{eqn:chios}) and Noether's formula \cite{BPV} we have that 
 $$
 \frac{1}{12}(K_S^2+c_2(S)) = \chi(\cO_S)=1-q+p_g=1.
 $$
Note that $c_2(S)=\chi(S)=2+b$, where $b=b_2=12$ and $b_1=b_3=0$. Therefore
$K_S^2=10-b=-2$. Now (\ref{eqn:o1}) says that
 $$
 -2= K_S^2= \frac{(4-m_1)^2}{m_1} - m_2-\ldots - m_b  \leq \frac{(4-m_1)^2}{m_1} - 11,
 $$
using that $g=3$. Therefore $(4-m_1)^2 \geq 9m_1$, which is rewritten as $(m_1-16)(m_1-1)\geq 0$. 

If $m_1\geq 16$, then  the curve $D_1$ of genus $g=3$ has self-intersection $D_1^2\geq 2g+1$. 
The argument of \cite[Theorem 32]{MRT} concludes that $b\leq 2g+3$. 
This is a contradiction since $g=3$ and $b=12$.

Therefore we have that $m_1=1$. So
 $$
 K_S= 3 D_1-D_2 - \ldots - D_b
 $$
and $K_S^2=-2=9-m_2-\ldots -m_b \leq 9-11=-2$. Therefore there must be equality and $m_2=\ldots =m_b=1$.
The basis $\{D_1,D_2,\ldots ,D_b\}$ is a diagonal basis of $H_2(S,\ZZ)$.
Now we try to reconstruct $S$ in ``reverse''. Let $H,E_2,\ldots, E_b \in H_2(S,\ZZ)$ be defined by the equalities:
 \begin{align*}
 D_1 &=  10 H- 3E_2-\ldots -3E_b \, ,\\
 D_j &=  (3H -E_2-\ldots - E_b) + E_{j}, \qquad j=2,\ldots, b.
 \end{align*}
This is solved as:
 $$
 \begin{array}{rl}
  H &=   10 D_1 - 3D_2-  \ldots - 3D_b \, ,\\
 E_j &=  
 3D_1 +D_j -\sum\limits_{k=2}^b D_k \, , \ j= 2,\ldots, b, \\
 K_S &= - 3H + \sum\limits_{k=2}^b E_k \, .
 \end{array}
 $$
The following self-intersections are easily computed:
 \begin{align*}
&  H^2=1, \quad & & H\cdot E_j=0, \ j=2,\ldots, b,\\
 & E_j^2=-1, \quad & & E_j \cdot E_k =0, \ j\neq k, \\
 & K_S\cdot H=-3, \quad &&K_S\cdot E_j=-1, \ j=2,\ldots, b, \\ 
&  D_j\cdot E_j=0, \qquad   & & D_j\cdot E_k=1, \ j\neq k.
 \end{align*}

Now let us prove that the classes $H,E_2, \ldots, E_b$ are defined by effective divisors. First
$\chi(H)= 1 + \frac{H^2-K_S\cdot H}{2} = 3$.
Also $h^0(K_S-H)=0$ since $K_S-H\leq K_S$ and $h^0(K_S)=0$.
Thus it must be $h^0(H)\geq 3$ and $H$ is effective.
Next $\chi(E_j)=1+\frac12 (E_j^2-K_S\cdot E_j)=1$. 
Also $h^0(K_S-E_j)=0$ since $(K_S-E_j) =-D_j <0$. Therefore $h^0(E_j)\geq 1$ and $E_j$ is also effective.

Next note that $K_S+D_j=E_j$. Consider the long exact sequence in cohomology associated to the exact sequence
$$
 0\to \cO(K_S)\to \cO(K_S+D_j)\to \cO_{D_j}(K_{D_j}) \to 0.
 $$ 
 As $H^0(K_S)=H^1(K_S)=0$, we have that $h^0(E_j)=h^0(K_S+D_j)=h^0(\cO_{D_j}(K_{D_j}))=1$, since 
 $D_j$ is an elliptic curve. Also $h^2(E_j)=h^0(K_S-E_j)=0$, and hence $h^1(E_j)=0$ since $\chi(E_j)=1$.

Consider now the exact sequence
 \begin{equation}\label{eqn:oo2}
 0 \to \cO \to \cO(D_2+\ldots + D_b) \to \bigoplus_{j=2}^b \cO_{D_j}(D_j) \to 0,
 \end{equation}
which holds since $D_j$ are disjoint. As $D_j^2=-m_j=-1$, and $D_j$ is an elliptic curve,
we have $h^0(\cO_{D_j}(D_j))=0$ and $h^1(\cO_{D_j}(D_j))=1$. Therefore
(\ref{eqn:oo2}) and the fact that $h^1(\cO)=h^2(\cO)=0$ implies that 
$h^0(D_2+\ldots + D_b)=1$ and $h^1(D_2+\ldots+D_b)= b-1=11$.
Using that $3D_1 \equiv D_2+\ldots + D_{b-1}+E_b$, we have an exact sequence
 $$
 0 \to \cO(E_b) \to \cO(3D_1) \to \bigoplus_{j=2}^{b-1} \cO_{D_j}(E_b)  \to 0.
 $$
As $D_j\cdot E_b=1$ and $D_j$ is elliptic, we have $h^0( \cO_{D_j}(E_b) )=1$. Then
$h^0(3D_1)=b-1=11$, using $h^0(E_b)=1$ and $h^1(E_b)=0$ computed before. 

Now we compute $h^0(3D_1)$ in a different way. We have exact sequences:
\begin{align*}
 & 0 \to \cO  \to  \cO(D_1) \to \cO_{D_1}(D_1) \to 0 \\
 & 0 \to \cO(D_1)  \to  \cO(2D_1) \to \cO_{D_1}(2D_1) \to 0 \\
 & 0 \to \cO(2D_1) \to  \cO(3D_1) \to \cO_{D_1}(3D_1) \to 0 
\end{align*}
so 
 \begin{align}\label{eqn:ooo}
 h^0(3D_1)&\leq h^0(2D_1)+ h^0(\cO_{D_1}(3D_1))  \nonumber\\ &\leq 
h^0(D_1)+ h^0(\cO_{D_1}(2D_1))+ h^0(\cO_{D_1}(3D_1)) \nonumber\\ & \leq
1+h^0(\cO_{D_1}(D_1)) +h^0(\cO_{D_1}(2D_1)) +h^0(\cO_{D_1}(3D_1)).
\end{align}
We use Clifford's theorem \cite[p.\ 107]{ACGH} that says that for a curve of genus $g\geq 1$
and a divisor $D$ of degree $0\leq d \leq 2g-2$, we have $h^0(D) \leq \left[\frac{d}{2}\right]+1$.
Applying this to the curve $D_1$, we have $h^0(\cO_{D_1}(D_1))\leq 1$, 
$h^0(\cO_{D_1}(2D_1))\leq 2$, and $h^0(\cO_{D_1}(3D_1))\leq 2$, recalling that $D_1^2=1$.
Therefore (\ref{eqn:ooo}) says that it $h^0(3D_1)\leq 1+1+2+2=6$. This is a contradiction with
the previous computation of $h^0(3D_1)$.

\medskip

\section{The second Stiefel-Whitney class of the Smale-Barden manifold}

We end up with a proof of the last comments in the Introduction. 

We start by computing the Stiefel-Whitney class $w_2(M)$ of the $5$-manifold $M$ constructed in Corollary
\ref{cor:spin}. This manifold is a Seifert circle bundle $\pi:M\to X'$ over the cyclic $4$-orbifold constructed
with the manifold $X=\CP^2\#11\overline{\CP}{}^2$ of Section \ref{sec:3}, 
ramified over the curves $\tilde C_1,\ldots, \tilde C_{11},\tilde C_{12}=\tilde G$, with multiplicities $m_i=p^i$,
$i=1,\ldots, 12$, and $p$ a fixed prime.

\begin{remark}
 We could use other powers of $p$ for the $m_i$'s, as long as they are distinct. Also we can use $m_i=p^i\tilde m_i$,
 with $\gcd(\tilde m_i,p)=1$. However, for the computations below we stick to our choice $m_i=p^i$.
\end{remark}

The Seifert bundle $\pi:M\to X'$ is determined by the Chern class
 $$
 c_1(M/X')=c_1(B)+\sum \frac{b_i}{m_i} [\tilde C_i],
 $$
where $b_i$ are called the orbit invariants (they should satisfy $\gcd(b_i,m_i)=1$), and
$B$ is a suitable line bundle over $X$. By Theorem \ref{thm:12}, we have to impose the condition 
that the cohomology class
 $$
 c_1(M/e^{2\pi i \mu})=\mu \, c_1(B) +\sum b_i\frac{\mu}{m_i}[\tilde C_i] =
 p^{12} c_1(B) +\sum b_i p^{12-i} [\tilde C_i] 
 $$
is primitive and represented by some orbifold symplectic form $[\hat \o]$, being $\mu=\lcm(m_i)=p^{12}$. 
The proof of \cite[Lemma 20]{MRT} shows that in order to ensure this
we can take $b_i=1$ and a class $a=c_1(B) \in H^2(X,\ZZ)$ with
$a=\sum a_i  [\tilde C_i]$ and $\gcd(pa_1+1, p^2a_2+1)=1$.
Certainly, with this condition on the class $a$ and the numbers $b_i$, the Chern class is 
$$
c_1(M/e^{2\pi i \mu})=\sum p^{12-i}( p^i a_i + 1)[\tilde C_i]
$$
so it is primitive.

Let us see that we can ensure that $\gcd(pa_1+1, p^2a_2+1)=1$.
Take any $a_2 \in \ZZ$  and write 
$$
p^2a_2 +1 = \prod_{j=1}^l q_j^{m_j}
$$
with $q_j$ different primes. The condition $\gcd(pa_1+1, p^2a_2+1)=1$ is equivalent to the condition 
that $q_j$ does not divide $pa_1 + 1$ for all $j$. Since $q_j$ and $p$ are coprime, by Bezout's identity 
we have $1+ \a p= \b q_j$ for some $\a, \b \in \ZZ$.
In fact, all the numbers $\a, \b$ satisfying that condition are of the form 
$(\a,\b)=(\a_j + t q, \b_j +tp), t \in \ZZ$.
Applying this to $q_1, \ldots, q_l$ we get that the condition is that $a_1$
does not belong to the set 
 \begin{equation}\label{eqn:A}
 A=\{\a_1 + tq_1\, | \, t \in \ZZ\} \cup \dots \cup \{\a_l + t q_l\, |\, t \in \ZZ\}.
 \end{equation}
The set $A$ above is a union of arithmetic progressions of ratios $q_1, \dots, q_l$ which are
different primes. By the chinese remainder theorem there is $\alpha \pmod{\prod q_i}$
so that $\alpha\equiv \alpha_i \pmod q_i$ for all $i$. Therefore the set $\ZZ-A$ modulo $\prod q_i$
contains $\phi(\prod q_i)=\prod (q_i-1)$ elements. In particular it is infinite and of positive density.

The conclusion is that possible choices of $a_1, \dots , a_{12}$ 
consist of choosing freely $a_2, \ldots , a_{12} \in \ZZ$, and then choose $a_1 \in \ZZ - A$.

\begin{proposition}
If $p=2$ then $M$ is spin. If $p>2$ then we can arrange $c_1(B)$ and $b_i$ so that $M$ is spin or non-spin.
\end{proposition}

\begin{proof}
If $p=2$, then \cite[Proposition 13]{MT2} says that the map $\pi^*:H^2(X,\ZZ_2) \to H^2(M,\ZZ_2)$ is zero,
since the $[\tilde C_i]$ are in $\ker \pi^*$ and they span the cohomology. By formula (3) in \cite{MT2}, we have
 $w_2(M)=\pi^*(w_2(X)+\sum b_i [\tilde C_i]+c_1(B))=0$.
 
 If $p>2$, then \cite[Proposition 13]{MT2} says that the map $\pi^*:H^2(X,\ZZ_2) \to H^2(M,\ZZ_2)$ 
 has one-dimensional kernel spanned by $c_1(B)+ \sum b_i [\tilde C_i]$. Formula (2) in \cite{MT2}
 says that $w_2(M)=\pi^*(w_2(X))$. As $X=\CP^2\#11\overline{\CP}{}^2$, $w_2(X)=H+E_1+\ldots + E_{11}=
 \tilde C_1+\ldots+\tilde C_{12} \pmod{2}$. The manifold $M$ is spin or non-spin according to 
 whether $c_1(B)+ \sum b_i [\tilde C_i]$ is proportional to $\sum [\tilde C_i]$ or not$\pmod2$. 
 We can arrange that $M$ is non-spin by taking say $a_2$ odd. To get $M$ spin we need to take
 all $a_2,\ldots, a_{12}$ even, and then choose $a_1\in \ZZ-A'$, where $A'$ is defined
 as (\ref{eqn:A}) but including also $q_0=2$, $\a_0=0$ (note that all $q_j$ are odd in this case).
\end{proof}

\begin{remark}\label{rem:Roger}
We end up with the proof that the symplectic manifold produced in \cite{MRT} does not admit a complex structure.
That manifold $Z$ is constructed as follows: take the $4$-torus $\TT^4$ with coordinates $(x_1,x_2,x_3,x_4)$, 
and take $2$-tori $T_{12},T_{13},T_{14}$ along the directions $(x_1,x_2),(x_1,x_3),(x_1,x_4)$, respectively. We
arrange these $2$-tori to be disjoint and make them symplectic. Now perform Gompf connected sums \cite{Gompf}
with $3$-copies of the rational elliptic surface $E(1)$ along a fiber $F$. This produces 
 $$
 Z'=\TT^4\#_{T_{12}=F} E(1)\#_{T_{13}=F} E(1)\#_{T_{14}=F} E(1).
 $$
Now we blow-up twice to get $Z=Z'\# 2\overline{\CP}{}^2$. Then $Z$ is simply connected and 
it has $b_2(Z)=36$ and contains $36$ disjoint surfaces,
of which $31$ of genus $1$ and negative self-intersection, and two of genus $2$ and three of genus $3$, all of 
positive self-intersection (see Theorem 23 in \cite{MRT}). Then $b_2^+(Z)=5$ and $b_2^-(Z)=31$.

Suppose that $Z$ admits a complex structure. First, for a complex manifold $b_2^+=1+p_g$ and $b_2^-=h^{1,1}+p_g-1$,
thus $b_2^- \geq b_2^+ -1$, as $h^{1,1}\geq 1$. 
This implies that the orientation of $Z$ has to be the same as a complex manifold. Also $p_g=4$
and $h^{1,1}=28$. Using Noether's formula,
 $$
  \frac{K_Z^2+c_2}{12}=\chi(\cO_Z)=1-q+p_g=5.
   $$
As $c_2=\chi(Z)=38$, we get $K_Z^2=22$. As $K_Z^2>9$, the Enriques classification \cite[page 188]{BPV} implies that $Z$ is
of general type.

Next we use the Seiberg-Witten invariants of $Z$. For a minimal surface $X$ of general type, the only Seiberg-Witten basic classes
are $\pm K_X$ (see Proposition 2.2 in \cite{FM}). If $Z$ is the blow-up of $X$ at $s$ points, then the basic classes of $Z$
are $\kappa_Z=\pm K_X \pm E_1 \pm \ldots \pm E_s$, where $E_1,\ldots, E_s$ are the exceptional divisors. Note that
$\kappa_Z^2=K_X^2-s=K_Z^2=22$. 

Now we compute the Seiberg-Witten basic classes of $Z$.
The only Seiberg-Witten basic class of $\TT^4$  is $\kappa=0$. The Seiberg-Witten basic classes of a Gompf connected sum
along a torus can be found in \cite[Corollary 15]{Park}. Using the relative Seiberg-Witten invariant of $E(1)$ in \cite[Theorem 18]{Park},
we have that the Seiberg-Witten invariants satisfy
 $$
  SW_{X\#_F E(1)}=SW_X\cdot (e^F+ e^{-F}),
  $$ 
for a $4$-manifold $X$. Therefore the basic
classes of $Z'$ are only $\kappa'=\pm T_{12}\pm T_{13}\pm T_{14}$.
When blowing-up, the basic classes of $Z$ are $\kappa_Z=\pm T_{12}\pm T_{13}\pm T_{14} \pm E_1\pm E_2$. Then
$\kappa_Z^2=-2$, which is a contradiction.
\end{remark}


\begin{thebibliography}{33}

\bibitem{ABCKT} \textsc{J.~Amor\'os, M.~Burger, K.~Corlette, D.~Kotschick, D.~Toledo}, 
{\it Fundamental groups of compact K\"ahler manifolds}, Math.\ Surveys and Monographs \textbf{44}, Amer.\ Math.\ Soc., 1996.

\bibitem{ACGH} \textsc{E. Arbarello, M. Cornalba, P. Griffiths, J. Harris}, 
{\it Geometry of Algebraic Curves, Volume I}. Springer-Verlag, 1985.

\bibitem{Barden} \textsc{Barden}, \emph{Simply Connected Five-Manifolds}, Annals of Mathematics, 1965.

\bibitem{BPV} \textsc{W. Barth, C. Peters, A. Van de Ven}, {\it Compact Complex Surfaces}, Springer, 1984.

\bibitem{BFMT} \textsc{I. Biswas, M. Fern\'andez, V. Mu\~noz, A. Tralle},  
{\it On formality of orbifolds and Sasakian manifolds}, J. Topology, 9 (2016), 161-180.

\bibitem{BG} \textsc{C. Boyer, K. Galicki}, {\it Sasakian Geometry}, Oxford Univ. Press, 2007.

\bibitem{Silva2} \textsc{A. Cannas da Silva}, {\it Lectures on Symplectic Geometry}, Lecture Notes in Mathematics, 1764. Springer-Verlag, 2008.

\bibitem{CNY} \textsc{B. Cappelletti-Montano, A. de Nicola, I. Yudin}, {\it Hard Lefschetz theorem for Sasakian manifolds}, Arxiv:1306.2896.

\bibitem{CNMY} \textsc{B. Cappelletti-Montano, A. de Nicola, J.C. Marrero, I. Yudin}, {\it Examples of
compact K-contact manifolds with no Sasakian metric}, Arxiv:1311.3270.


\bibitem{DGMS}
\textsc{P.~Deligne, P.~Griffiths, J.~Morgan and D.~Sullivan}, {\it Real homotopy theory of K\"ahler manifolds}, Invent.\ Math.\ {\bf 29} (1975), 245-274.

\bibitem{FM} \textsc{R. Friedman, J.  Morgan}, {\it Algebraic surfaces and Seiberg-Witten invariants}, arxiv:alg-geom/9502026.
 
\bibitem{Gay-Stip} \textsc{D. Gay, A. Stipsicz}, \emph{Symplectic surgeries and normal surface singularities}, Algebr. Geom. Topol. Volume 9, Number 4 (2009), 2203-2223

\bibitem{Gay-Mark} \textsc{D. Gay, T. Mark}, \emph{Convex plumbings and Lefschetz fibrations}, Journal of Symplectic Geometry, 2013

\bibitem{Gompf} \textsc{R. Gompf}, {\it A new construction of symplectic manifolds}, Annals of Math. (2) 142 (1995), 537-696.

\bibitem{HT} \textsc{B. Hajduk, A. Tralle}, {\it On simply-connected compact K-contact non-Sasakian manifolds}, J. Fixed Point Theory Appl. 16 (2014), 229-241.

\bibitem{HD} \textsc{A. Haefliger, Q. N. Du}, {\it Appendice: une pr\'esentation du groupe fondamental d'une orbifold}, 
Ast\'erisque {\bf 116} (1984), 98--107, Transversal structure of foliations (Toulouse, 1982). 

\bibitem{JK} \textsc{J.-B. Jun, U.K. Kim}, \emph{On 3-dimensional almost contact metric manifolds},
Kyungpook Math. J. 34 (2) (1994), 293 – 301.

\bibitem{Kollar} \textsc{J. Koll\'ar}, {\it Circle actions on simply connected $5$-manifolds}, Topology, 45 (2006) 643-672.

\bibitem{Le} \textsc{S. Lefschetz}, {\it L'analysis situs et la g\'eom\'etrie alg\'ebrique}, Gauthier-Villars (1924).

\bibitem{Lin} \textsc{Y. Lin}, {\it Lefschetz contact manifolds and odd dimensional symplectic geometry}, Arxiv:1311.1431.

\bibitem{MRT} \textsc{V. Mu\~noz, J. Rojo, A. Tralle}, {\it Homology Samale-Barden manifolds with K-contact but not Sasakian structures}, Internat.\  Math.\ Research Notices (2018), to appear.

\bibitem{MT}\textsc{V. Mu\~noz, A. Tralle}, {\it Simply connected K-contact and Sasakian manifolds of dimension 7}, Math. Z. 281 (2015), 457-470.

\bibitem{MT2}\textsc{V. Mu\~noz, A. Tralle}, {\it On the classification of Smale-Barden manifolds with Sasakian structures}, preprint, 2020.

\bibitem{Ruk} \textsc{P. Rukimbira}, \emph{Chern-Hamilton conjecture and K-contactness}, Houston J. Math, 1995.

\bibitem{Smale} \textsc{S. Smale}, \emph{On the Structure of 5-Manifolds}, Annals of Mathematics, 1962.

\bibitem{OT} \textsc{J. Oprea, A. Tralle}, Symplectic Manifolds with no Kaehler structure, 1997, Springer.

\bibitem{Park}
\textsc{D. Park}, {\it A Gluing formula for the Seiberg-Witten invariant along $T^3$}, Michigan Math. J. 50 (2002), 593-611.

\bibitem{Tievsky} \textsc{A. Tievsky}, {\em Analogues of K\"ahler geometry on Sasakian manifolds}, Ph.D.\ Thesis, MIT, 2008.

\bibitem{Thu}  \textsc{W. Thurston}, {\em The geometry and topology of $3$-manifolds}, Mimeographed Notes, Princeton University, 1979.


\end{thebibliography}
\end{document}